\documentclass[11pt]{article}

\usepackage[margin=1in]{geometry}
\usepackage[T1]{fontenc}
\usepackage{lmodern}
\usepackage{microtype}
\usepackage{amsmath,amssymb,amsthm,mathtools}
\usepackage{booktabs,array,longtable}
\usepackage{xcolor}
\usepackage{enumitem}
\usepackage{hyperref}
\usepackage[authoryear,round]{natbib}
\usepackage{tikz}
\usetikzlibrary{arrows.meta,positioning}
\usepackage{seqsplit}
\usepackage{fvextra}
\usepackage{xurl}

\hypersetup{
  colorlinks=true,
  linkcolor=blue!50!black,
  citecolor=blue!50!black,
  urlcolor=blue!55!black,
  pdftitle={A Positive-Measure Geometric Lower Bound for the Barzilai--Borwein Method}
}

\newtheorem{theorem}{Theorem}[section]
\newtheorem{proposition}[theorem]{Proposition}
\newtheorem{lemma}[theorem]{Lemma}

\theoremstyle{definition}

\newtheorem{certificate}[theorem]{Computer-assisted certificate}
\theoremstyle{remark}

\newcommand{\R}{\mathbb{R}}
\newcommand{\Spp}{\mathbb{S}_{++}}
\newcommand{\one}{\mathbf{1}}
\newcommand{\norm}[1]{\left\lVert #1\right\rVert}
\newcommand{\abs}[1]{\left|#1\right|}

\newcommand{\diag}{\operatorname{diag}}
\newcommand{\intt}{\operatorname{int}}

\newcommand{\rhoSp}{\rho}

\newcommand{\cmin}{\underline\rho}
\newcommand{\cmax}{\overline\rho}

\setlist{itemsep=0.35em,topsep=0.45em}
\title{\bfseries Barzilai–Borwein Fails Superlinear Convergence on an Open Set of Quadratics for Every Dimension $n\geq 4$}
\author{Dawei Li \qquad Xiaotian Jiang \qquad Mingyi Hong \\ University of Minnesota \\
   \texttt{\{li004678,jian0851,mhong\}@umn.edu}}
\date{July 14, 2026}

\begin{document}
\maketitle

\begin{abstract}
Barzilai--Borwein (BB) method has shown strong practical performance in continuous optimization, yet its convergence dynamics remains poorly understood. In particular, a central unresolved question is whether BB converges superlinearly for almost every strictly convex quadratic problem and initialization. We provide a negative answer to this question. 
Specifically, for every finite dimension $n\geq4$, we construct a nonempty open, hence positive-Lebesgue-measure, family of strictly convex quadratic problems and initial points for which the long Barzilai--Borwein method (BB1) converges but cannot converge root-superlinearly.  More precisely, with the explicit constants
\[
  \cmin=10^{-6},\qquad \cmax=0.61,
\]
every spectral component of the gradient is bounded above and below by the corresponding geometric sequence. Consequently, the gradient norm and the energy norm of the error satisfy two-sided geometric estimates with the same rates, while the objective gap satisfies the corresponding estimates with squared rates. In particular, all three quantities are bounded below by geometric sequences, ruling out superlinear convergence.
The construction is highly nontrivial, based on a computer-assisted proof of a nonresonant, attracting seven-cycle of the projectivized BB dynamics in dimension four.
\end{abstract}


\section{Introduction}
\label{sec:introduction}

Gradient descent is one of the oldest and most widely used methods in continuous optimization. Each iteration requires only a gradient and a few vector operations, making it attractive in large-scale problems. Its performance, however, can deteriorate on ill-conditioned settings: exact line search may produce zigzag iterates, whereas a good constant step requires spectral information that is rarely available in advance. 

The Barzilai--Borwein (BB) method \citep{BarzilaiBorwein1988} addresses this difficulty without changing the negative-gradient direction. Instead,  it extracts a scalar curvature estimate from the two most recent iterates. This change of step-size rule, rather than search direction, can substantially improve the practical behavior of a basic gradient iteration \citep{Fletcher2005,BirginMartinezRaydan2014}.

For a differentiable objective $f:\mathbb R^n\to\mathbb R$, write
$g_k=\nabla f(x_k)$, $s_{k-1}=x_k-x_{k-1}$, and $y_{k-1}=g_k-g_{k-1}$. The two classical BB step sizes are \citep{BarzilaiBorwein1988}
\begin{equation}
  \alpha_k^{\mathrm{BB1}}
  =\frac{s_{k-1}^{\mathsf T}s_{k-1}}
         {s_{k-1}^{\mathsf T}y_{k-1}},
  \qquad
  \alpha_k^{\mathrm{BB2}}
  =\frac{s_{k-1}^{\mathsf T}y_{k-1}}
         {y_{k-1}^{\mathsf T}y_{k-1}},
  \qquad
  x_{k+1}=x_k-\alpha_k g_k.
  \label{eq:intro-bb-steps}
\end{equation}
The first choice makes the scalar matrix $(\alpha_k^{\mathrm{BB1}})^{-1}I$ the least-squares solution of the secant equation; the second makes $\alpha_k^{\mathrm{BB2}}I$ a least-squares approximation of the inverse Hessian. They are therefore often called the long and short BB steps, respectively. 
The formulation of the two BB steps are different, but they possess similar properties in training.
BB runs at essentially the cost of gradient descent while carrying curvature information without forming a matrix and, in its basic form, without performing a line search. On a quadratic, the reciprocal BB step is a Rayleigh quotient of the Hessian; this is also the origin of the name \emph{spectral gradient method}. 

Despite this simple construction, BB generates nonlinear, history-dependent dynamics even on quadratic objectives. A longstanding open question is whether BB converges superlinearly for almost every strictly convex quadratic problem and initialization, or whether failure of
root-superlinear convergence can persist on a set of positive measure rather
than only on exceptional instances. More precisely, the question is whether
\[
  \lim_{k\to\infty}\norm{g_k}^{1/k}=0
\]
for almost every strictly convex quadratic problem and initialization. We aim to resolve this question in this paper.

\paragraph{Significance and applications.}
The BB step has become an algorithmic primitive rather than a single method.
Raydan combined it with a non-monotone globalization strategy for large-scale
unconstrained optimization \citep{Raydan1997}; related non-monotone
frameworks were developed in \citet{GrippoSciandrone2002} and
\citet{ZhangHager2004}, and the spectral projected gradient method brought
the same step selection to optimization over convex sets
\citep{BirginMartinezRaydan2000,BirginMartinezRaydan2001,DaiFletcher2005PBB,AndreaniBirginMartinezYuan2005}.
BB-type spectral steps now appear in support-vector-machine training
\citep{SerafiniZanghiratiZanni2005}, sparse reconstruction and compressed
sensing
\citep{FigueiredoNowakWright2007,WrightNowakFigueiredo2009,VanDenBergFriedlander2009,WenYinGoldfarbZhang2010},
image restoration \citep{WangMa2007,BonettiniZanellaZanni2009}, nonnegative
matrix factorization \citep{HuangLiuZhou2015}, optimization with
orthogonality constraints \citep{WenYin2013}, Riemannian optimization
\citep{IannazzoPorcelli2018}, stochastic and variance-reduced methods
\citep{TanMaDaiQian2016}, distributed optimization with locally computed
step sizes \citep{GaoLiuDaiHuangYang2022}, and PDE-constrained optimization
in Hilbert space \citep{AzmiKunisch2020,AzmiKunisch2022}; more recent
safeguarded and adaptive constructions seek global guarantees on broader
function classes \citep{BurdakovDaiHuang2019,ZhouMaYang2026}.
These descendants alter the pure iteration in essential ways, such as projection, regularization, line searches, stochastic averaging, safeguarding, but all of them inherit its step-size mechanism. Understanding the unmodified deterministic iteration is therefore valuable both in its own right and as a foundation for interpreting its many descendants.

\paragraph{Earlier convergence and rate theory.}
The known convergence theory of BB mostly focuses on low-dimensional problems. \cite{BarzilaiBorwein1988} proved $R$-superlinear convergence for two-dimensional strictly convex quadratics. \cite{Raydan1993} established global convergence for arbitrary finite-dimensional strictly convex quadratics, and \cite{DaiLiao2002} proved an $R$-linear rate in
arbitrary dimension, together with a local consequence for sufficiently smooth
nonquadratic objectives near a nondegenerate minimizer. 
In dimension two, \citet{Dai2013} related the rate to the initial spectral balance and showed that superlinear behavior holds for almost every initialization, while linear behavior occurs on an exceptional set. \cite{DaiFletcher2005} developed an asymptotic framework for gradient methods with retarded spectral information and, by computing the asymptotics of BB and its relatives, observed a transition from superlinear to linear behavior at some dimension $n\geq 4$ depending on the method. Their statements are asymptotic computations and numerical observations, and no rigorous lower-rate mechanism in dimension $n\geq 4$ was identified.

For general dimension, rate bounds have been substantially less informative about the observed fast behavior. A recent refinement by \citet{LiSun2021} proves, for BB1 on an arbitrary finite-dimensional strongly convex quadratic with condition number $\kappa=\lambda_n/\lambda_1$, an $R$-linear upper factor $1-1/\kappa$. This is comparable in condition-number scaling to the exact-line search steepest-descent factor. The same work gives a lower-rate example supported only on the two extreme eigenspaces, with exact factor $(\kappa-1)/(\kappa+1)$, ruling out superlinear convergence. However, for $n\geq 2$ its initial point lies in a Lebesgue-null set, so an arbitrarily small perturbation could possibly accelerate the convergence; the same paper explicitly raises the convergence rate under generic initialization as an open question.

Parallel lines of work have sought either to globalize BB or to design spectral rules with more transparent asymptotics. These include gradient methods with retards \citep{FriedlanderMartinezMolinaRaydan1998}, alternate and cyclic BB schemes \citep{Dai2003,DaiHagerSchittkowskiZhang2006}, convex combinations of the long and short BB steps \citep{DaiHuangLiu2019}, and steps engineered for two-dimensional termination or favorable spectral limits \citep{Yuan2006,HuangDaiLiu2021,HuangDaiLiuZhang2022}. This literature gives considerable evidence that low-dimensional spectral mechanisms govern the performance of many gradient methods.
What has been missing for the original BB1 method is an answer to the following question: 

\begin{center}
\fbox{%
  \parbox{0.9\linewidth}{%
    \centering
    In dimension $n\geq4$, is every linear-rate BB1 example an unstable exception, or does genuinely full-dimensional linear-rate behavior persist on a robust family?
  }%
}
\end{center}

\paragraph{Contribution of the present lower bound.}
This paper answers this question for BB1. For every finite dimension $n\geq4$ and every linear term $b\in\mathbb R^n$, we construct a nonempty open set $\mathcal A_n\subset\mathbb S_{++}^n$ of simple-spectrum matrices such that, for every $A\in\mathcal A_n$, there is a nonempty open set $\mathcal X_n(A)$ of initial points for which the admissibly initialized BB1 iteration is well defined and converges to $x_*=A^{-1}b$.  More precisely, if $\Pi_i(A)$ denotes the orthogonal spectral projector associated with the $i$th eigenvalue, then

\begin{equation}
  10^{-6k}\,\lVert\Pi_i(A)g_0\rVert
  \;\leq\;
  \lVert\Pi_i(A)g_k\rVert
  \;\leq\;
  0.61^k\,\lVert\Pi_i(A)g_0\rVert,
  \qquad i=1,\ldots,n,\quad k\geq0.
  \label{eq:intro-main-bound}
\end{equation}
Here ``admissibly initialized'' means that the first step is the inverse gradient Rayleigh quotient,
\[
  \alpha_0=\frac{g_0^{\mathsf T}g_0}{g_0^{\mathsf T}Ag_0},
  \qquad x_1=x_0-\alpha_0g_0,
\]
so the orbit is an actual BB1 trajectory rather than an arbitrarily prescribed
pair of projective states.  The joint set $\Omega_n
  =\{(A,x_0):A\in\mathcal A_n,\ x_0\in\mathcal X_n(A)\}$
is open and nonempty in $\mathbb S_{++}^n\times\mathbb R^n$.  It therefore has positive Lebesgue measure, as does every initial-point fiber $\mathcal X_n(A)$.  Summing \eqref{eq:intro-main-bound} over the orthogonal spectral components gives the same two-sided geometric estimate for the gradient norm and the energy norm of the error; the objective gap satisfies the corresponding squared estimate.  In particular,
\[
  \liminf_{k\to\infty}\lVert g_k\rVert^{1/k}\geq10^{-6}>0,
\]
so none of these trajectories is root-superlinearly or $Q$-superlinearly convergent. An almost-everywhere superlinearity statement on the joint problem-data space is therefore false in every dimension $n\geq 4$; in particular, the question of \citet{LiSun2021} about generic initialization is answered in the negative on an open set of matrices.

The result strengthens the earlier lower-rate picture in three ways.

\begin{enumerate}[label=\emph{(\roman*)},leftmargin=2.8em]
  \item \emph{Robustness.}  The slow behavior is not confined to a specially balanced endpoint subspace or to a zero-measure initialization. It persists under simultaneous perturbations of the spectrum and the initial point.

  \item \emph{A dynamical mechanism.}  The proof identifies a nonresonant, attracting period-seven orbit of the projectivized four-dimensional BB1 dynamics. A certified fifteen-step connection reaches its basin from the admissible-initialization manifold, and a spectral translation makes every radial multiplier contractive. Attraction of the projective cycle, rather than invariance of a coordinate subspace, creates the open basin and the lower geometric rate.

  \item \emph{Full-dimensional persistence.}  Four spectral modes remain active along the asymptotic cycle. A uniform transverse Floquet estimate then embeds the construction into each finite dimension $n>4$ while keeping every additional component nonzero and geometrically controlled.  The normalized spectral measure consequently approaches a seven-periodic measure rather than collapsing to the two endpoint modes familiar from exact steepest descent \citep{Akaike1959,Forsythe1968,PronzatoWynnZhigljavsky2006}.
\end{enumerate}

The periodic orbit is certified by interval arithmetic and a contraction argument; it is not inferred from a long floating-point simulation. This computer-assisted component establishes exact existence, nonresonance, and attraction inside explicit rational boxes; openness, positive measure, and persistence in higher dimension then follow from analytic stability arguments.

The theorem is an existence result, not a universal rate formula.  It
concerns BB1, not every BB variant.  The set $\mathcal A_n$ is a small open
neighborhood of one explicitly constructed spectrum---three separated
eigenvalues together with, for $n>4$, the remaining eigenvalues clustered in
a short interval near the top---so the phrase ``every finite dimension''
asserts that such a neighborhood exists for each $n$, not that its size is
uniform in $n$.  The constants $10^{-6}$ and $0.61$ are conservative
separation margins rather than claimed optimal asymptotic factors.  The
contribution is structural: it proves that the dimension-four transition from low-dimensional superlinear behavior to linear-rate behavior can be generated by a stable recurrent attractor and can occupy positive measure.  In this sense, the lower bound complements the classical and recent $R$-linear upper bounds.  The upper theory controls how slow BB1 can be; the present result shows that a positive linear root factor is actually realized robustly, and explains the orbit geometry that sustains it.

\paragraph{AI-Assisted Development.}
The sketch of the proof of the main result of this paper was obtained by GPT 5.6 Sol. GPT was given an equivalent problem formulation (component-wise quotient formulation) and asked to construct the lower bound example. The prompt is provided in Appendix \ref{app:audit-prompt}. The construction by GPT is highly non-trivial. Human verification and polishing were done afterwards. 

\section{The BB1 method and the main theorem}

Let $\Spp^n$ denote the open cone of real symmetric positive-definite
$n\times n$ matrices.  It is an open subset of the Euclidean space
$\mathbb S^n$, which has dimension $n(n+1)/2$.  Fix $A\in\Spp^n$ and
$b\in\R^n$, and consider
\begin{equation}\label{eq:quadratic}
  f_{A,b}(x)=\frac12x^TAx-b^Tx,
  \qquad x_*:=A^{-1}b,
  \qquad g(x):=\nabla f_{A,b}(x)=Ax-b.
\end{equation}
The unique minimizer is $x_*$ because $A\succ0$.

We study the long Barzilai--Borwein method, usually called BB1
\cite{BarzilaiBorwein1988}; classical convergence analyses on strictly convex
quadratics include \cite{Raydan1993,DaiLiao2002}.  The initialization is part of the theorem and is
therefore stated explicitly.  Given $x_0\neq x_*$, put $g_0=Ax_0-b$ and
\begin{equation}\label{eq:alpha0}
 \alpha_0:=\frac{g_0^Tg_0}{g_0^TAg_0},
 \qquad x_1:=x_0-\alpha_0g_0.
\end{equation}
For $k\geq1$, provided the quantities are defined, set
\begin{equation}\label{eq:bb1}
 \begin{aligned}
   s_{k-1}&:=x_k-x_{k-1},
   &y_{k-1}&:=g_k-g_{k-1},\\
   \alpha_k&:=\frac{s_{k-1}^Ts_{k-1}}{s_{k-1}^Ty_{k-1}},
   &x_{k+1}&:=x_k-\alpha_kg_k,
 \end{aligned}
 \qquad g_k:=Ax_k-b.
\end{equation}
The choice \eqref{eq:alpha0} is the inverse Rayleigh quotient of the initial gradient. It makes the first two reciprocal step sizes equal. 

\paragraph{Main Result.}
For a simple-spectrum matrix $A$, let
\[
  0<\lambda_1(A)<\cdots<\lambda_n(A)
\]
be its ordered eigenvalues and let $\Pi_i(A)$ be the orthogonal projector onto
the one-dimensional eigenspace corresponding to $\lambda_i(A)$.  The
orientation of an eigenvector is irrelevant because all conclusions below use $\norm{\Pi_i(A)g_k}$.

\begin{theorem}[Positive-measure lower rate for BB1]\label{thm:main}
Fix an integer $n\geq4$ and a vector $b\in\R^n$.  There exist a nonempty open
set $\mathcal A_n\subset\Spp^n$ consisting of simple-spectrum matrices and,
for every $A\in\mathcal A_n$, a nonempty open set
$\mathcal X_n(A)\subset\R^n\setminus\{x_*\}$ with the following properties.

For every $x_0\in\mathcal X_n(A)$, the BB1 iteration
\eqref{eq:alpha0}--\eqref{eq:bb1} is well-defined for every $k\geq0$, converges
to $x_*$, and satisfies, for every $i\in\{1,\ldots,n\}$ and every $k\geq0$,
\begin{equation}\label{eq:component-main}
 \cmin^k\norm{\Pi_i(A)g_0}
 \leq \norm{\Pi_i(A)g_k}
 \leq \cmax^k\norm{\Pi_i(A)g_0},
 \qquad \cmin=10^{-6},\quad\cmax=0.61.
\end{equation}
In particular, none of the initial spectral components is zero.

The union
\begin{equation}\label{eq:omega-def}
 \Omega_n:=\{(A,x):A\in\mathcal A_n,\ x\in\mathcal X_n(A)\}
 \subset \Spp^n\times\R^n
\end{equation}
is nonempty and open.  Consequently $\Omega_n$ has positive Lebesgue measure
in $\mathbb S^n\times\R^n$, and each fiber $\mathcal X_n(A)$ has positive
Lebesgue measure in $\R^n$.

Moreover, writing $e_k=x_k-x_*$ and
$\norm{e}_A=(e^TAe)^{1/2}$, one has
\begin{align}
 \cmin^k\norm{g_0}
 &\leq \norm{g_k}\leq \cmax^k\norm{g_0},
 \label{eq:gradient-main}\\
 \cmin^k\norm{e_0}_A
 &\leq \norm{e_k}_A\leq \cmax^k\norm{e_0}_A,
 \label{eq:energy-main}\\
 \cmin^{2k}\bigl(f(x_0)-f(x_*)\bigr)
 &\leq f(x_k)-f(x_*)
 \leq \cmax^{2k}\bigl(f(x_0)-f(x_*)\bigr).
 \label{eq:objective-main}
\end{align}
Thus
\begin{equation}\label{eq:root-lower}
 \liminf_{k\to\infty}\norm{g_k}^{1/k}\geq10^{-6},
 \qquad
 \liminf_{k\to\infty}\norm{e_k}_A^{1/k}\geq10^{-6},
\end{equation}
and these BB1 orbits are not root-superlinearly or $Q$-superlinearly
convergent.
\end{theorem}

The proof occupies Sections~\ref{sec:BB-recurrence}--\ref{sec:lift}.
The constants are conservative and are not claimed to be optimal. The result is an existence theorem for an open family; it does not assert the same lower bound for every positive spectrum.

\paragraph{Why the BB dynamics are difficult.}
Even for the quadratic objective \eqref{eq:quadratic}, BB1 generates a
nonlinear delayed dynamical system.
Let $0<\lambda_1\leq\cdots\leq\lambda_n$ be the eigenvalues of $A$, let
$v_1,\ldots,v_n$ be an orthonormal eigenbasis, and write
$g_k=\sum_i d_k^i v_i$.  Then for BB1 on this quadratic,

\begin{equation}
  \alpha_k
  =\frac{\sum_{j=1}^n(d_{k-1}^j)^2}
         {\sum_{j=1}^n\lambda_j(d_{k-1}^j)^2},
  \qquad
  d_{k+1}^i=(1-\lambda_i\alpha_k)d_k^i,
  \label{eq:intro-spectral-recurrence}
\end{equation}
equivalently,
\begin{equation}
  d_{k+1}^i
  =d_k^i\,
  \frac{\sum_{j=1}^n(\lambda_j-\lambda_i)(d_{k-1}^j)^2}
       {\sum_{j=1}^n\lambda_j(d_{k-1}^j)^2}.
  \label{eq:intro-delayed-recurrence}
\end{equation}

Equation \eqref{eq:intro-delayed-recurrence} displays the first source of
difficulty.  Diagonalization completely decouples fixed-step gradient
descent, but it does not decouple BB1. The update of the $i$th component $d^i$ from time $k$ to time $k+1$ depends on every component at time $k-1$. Thus, although the objective is diagonal in the eigenbasis, its spectral modes remain coupled through a delayed Rayleigh quotient.
This coupled and delayed feedback produces nonlinear behavior. Depending on
the Rayleigh quotient selected by the preceding gradient, the multiplier of a spectral component can be negative, close to zero, or greater than one in
absolute value. Consequently, components can change sign, high-eigenvalue components can grow, and different eigendirections may take turns dominating the gradient, so neither the objective value nor the gradient norm is monotone.
(On nonquadratic objectives, the situation is more delicate -- $s_{k-1}^{\mathsf T}y_{k-1}$ may fail to remain positive, and the unsafeguarded method can take excessively long steps or fail to converge even under strong convexity \citep{Fletcher2005,BurdakovDaiHuang2019} -- but they are not the subject of this paper.)

The second difficulty is determining the actual asymptotic rate. As a
component multiplier can be close to zero, a spectral component may become
arbitrarily small in a single iteration. 
An upper estimate $\lVert g_k\rVert\leq Cq^k$ therefore cannot distinguish genuinely linear convergence from faster convergence, such as superlinear.
Such an upper bound also does not determine which spectral modes survive, whether the normalized state approaches a fixed point or a cycle, or what lower asymptotic rate is actually attained. Ruling out superlinear convergence needs a lower bound, or an equally precise description of the asymptotic orbit. 

There is a third difficulty in answering an almost-everywhere question.  A
single specially constructed lower bound example could be an isolated exception and hence belong to a set of Lebesgue measure zero \citep{Dai2013,LiSun2021}. In practice, such an example may depart from the theoretical trajectory and lead to faster convergence rate due to numerical instability \cite{he2025new}. 
A robust lower bound example should additionally prove that the convergence rate persists under perturbations of both the spectrum and the initialization.
This is the aspect of BB dynamics for which the classical theory has been least complete.

\section{Outline of the proof}
\label{sec:outline}

The lower bound is not obtained by estimating a generic BB1 trajectory.
The proof constructs one rigorously certified orbit of a low-dimensional
dynamical system, verifies that this orbit stays uniformly away from every
resonance, and then lets stability and continuity enlarge the single orbit
into an open family.  Figure~\ref{fig:proof-map} displays the chain of
implications; this section walks through it once.

\paragraph{Reduction to a factor bound (\S\ref{sec:BB-recurrence}).}
On a quadratic, BB1 acts on each spectral component of the gradient through
the exact multiplier identity
$d_{k+1}^i=\bigl((a_k-\lambda_i)/a_k\bigr)\,d_k^i$, where
$a_k=\alpha_k^{-1}$ is the delayed Rayleigh quotient, so that
$\abs{d_k^i}=\abs{d_0^i}\prod_{t<k}\abs{a_t-\lambda_i}/a_t$.  All
conclusions of Theorem~\ref{thm:main} therefore follow from a single uniform
one-step inequality,
\begin{equation}\label{eq:outline-factor}
 10^{-6}\;\leq\;\frac{\abs{a_k-\lambda_i}}{a_k}\;\leq\;0.61
 \qquad\text{for every $i$ and every $k\geq0$:}
\end{equation}
the lower bound is a nonresonance condition ($a_k$ never comes too close to
an eigenvalue), the upper bound forces convergence.  Since the gradient
itself tends to zero, its scale is removed: the pair $(a_k,p_k)$ of the
reciprocal step size and the normalized squared spectral weights evolves
under a rational map $F_\lambda$ on $\R\times\Delta_{n-1}$, with the
admissible first step encoded by an initialization map $D_\lambda$.  If this
projective state approaches a periodic orbit that is separated from the
spectrum, all late multipliers in \eqref{eq:outline-factor} lie in a fixed
compact subinterval of $(0,1)$, and the finite initial segment is handled by
continuity.

\begin{figure}[t]
\centering
\includegraphics[width=\textwidth]{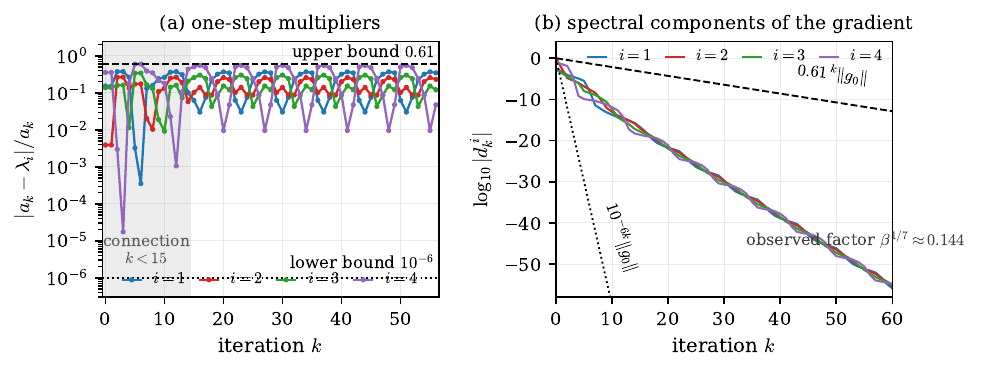}
\caption{Numerical realization of the factor bounds along the certified
orbit at the central shifted spectrum $\lambda^*$, computed in $200$-digit
arithmetic.  \emph{(a)} The one-step multipliers
$\abs{a_k-\lambda_i}/a_k$: after the fifteen-step connection (shaded) they
are seven-periodic and remain inside the certified band
$[10^{-6},\,0.61]$; the transient minimum, approximately
$1.7\times10^{-5}$, occurs at $k=3$.  \emph{(b)} The four spectral
components of the gradient decay in parallel between the two geometric
envelopes; the observed asymptotic per-step factor is
$\beta^{1/7}\approx0.144$, where $\beta\approx1.29\times10^{-6}$ is the
contraction of the raw coordinates over one seven-cycle.  The figure is an
illustration; the rigorous statements are those of
Certificate~\ref{cert:main}.}
\label{fig:factor-bounds}
\end{figure}

\paragraph{The certified periodic object (\S\ref{sec:certificate}).}
The construction is carried out in dimension four.  The unknowns are a free
eigenvalue $\bar\lambda_3$, a projective state $z$, and an admissible initial
weight $r$; they satisfy eight rational equations, namely
$F_{\bar\lambda}^{7}(z)=z$ (a seven-cycle) together with
$F_{\bar\lambda}^{15}(D_{\bar\lambda}(r))=z$ (an exact fifteen-step landing).
The landing block is forced by a dimension count: the projective state space
is four-dimensional while the admissibly initialized states form a
three-dimensional hypersurface, so an attracting cycle need not meet the
forward orbit of any genuine BB1 initialization.  A Newton-like map
$H(u)=u-RG(u)$ with an exact rational preconditioner $R$ is then evaluated in
directed interval arithmetic on a box $X$ of radius $10^{-55}$; the certified
inclusions $H(X)\subset\intt X$ and $\sup_X\norm{DH}_\infty<1$ make Banach's
theorem produce an exact root $u_*$
(Lemma~\ref{lem:exact-root}).  The same run certifies that every
normalization denominator (namely $Z_t$) is positive, that the cycle and the connection stay
separated from all eigenvalues, and that an explicit Lyapunov pair
$(P,Q_7)$ is positive definite.  Two logically distinct contractions appear
here: $H$ only certifies the root; attraction is proved separately.

\paragraph{Attraction, rates, and openness in dimension four
(\S\ref{sec:floquet}--\S\ref{sec:four-open}).}
The certified inequality $P-A_7^TPA_7\succ0$ gives $\norm{A_7}_P<1$, so the
seven-cycle is Schur stable and possesses an open basin
(Lemma~\ref{lem:nonlinear-attraction}).  The rate constants come from two
independent mechanisms.  A translation $\lambda\mapsto\lambda+4\one$ commutes
with the projective dynamics (Lemma~\ref{lem:translation}) and moves the
spectrum into $(4.99,\,8.01)$ with width below $3.02$; since every $a_k$ is a
convex combination of eigenvalues, this alone yields the upper factor
$3.02/4.99<0.61$.  The certified separation yields the lower factor
$10^{-5}/8.01>10^{-6}$.  Because $1$ is not an eigenvalue of $A_7$, the
implicit-function theorem continues the cycle to all nearby spectra, and
finite-time continuity of the fifteen-step map carries all nearby admissible
initializations into the trapping tube: one trajectory becomes an open
product $L_4\times P_4$ (Proposition~\ref{prop:n4-open}).

\paragraph{Every dimension, and the lift to BB1 problems
(\S\ref{sec:higher-n}--\S\ref{sec:lift}).}
For $n>4$, the extra eigenvalues are placed in $(7.49,7.51)$ with zero
initial weight, so the four-dimensional cycle survives on an invariant
simplex face.  The seven-step derivative is block upper triangular there, and
the certificate bounds every transverse multiplier by
$\tau(\nu)<0.049$, so the boundary cycle attracts in the new directions as
well; tilting the extra weights to small positive values and continuing in
all $n$ eigenvalues gives open sets $L_n\times P_n$ for each fixed finite
$n$ (Proposition~\ref{prop:n-open}), with the rate constants---though not
the neighborhood sizes---uniform in $n$.  Finally, three changes of
variables lift the projective family to problem data: the orthant polar map
$d_0^i=\sigma_ir\sqrt{p_{0,i}}$, the continuity of the spectral-projector
weights $p(A,g)$ on simple-spectrum matrices, and the diffeomorphism
$(A,x)\mapsto(A,Ax-b)$.  The resulting set $\Omega_n$ of problems and
initial points is open and nonempty, hence of positive Lebesgue measure, and
Theorem~\ref{thm:main} follows.

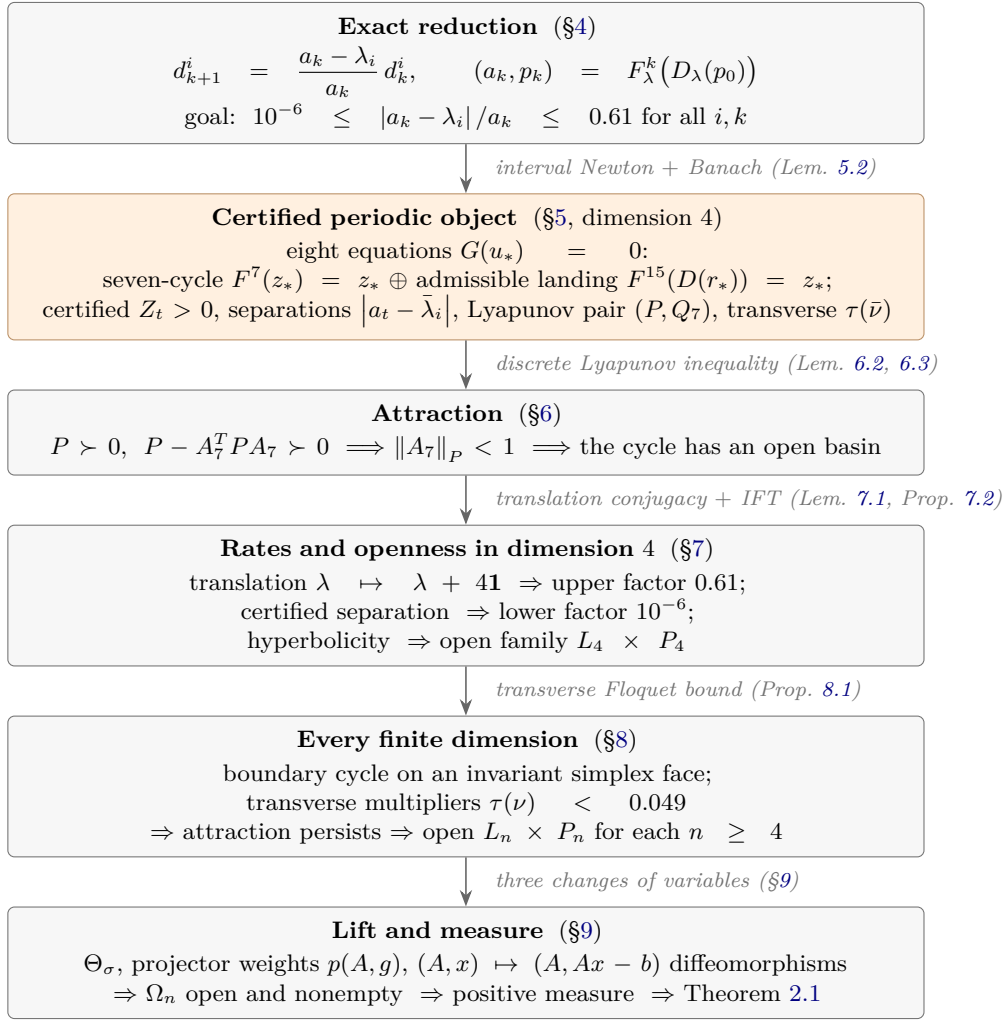
\begin{figure}[t]
\centering
\begin{tikzpicture}[>={Stealth[length=2.2mm]}, font=\footnotesize,
  node distance=6.5mm,
  stage/.style={draw=black!55, rounded corners=2.5pt, align=center,
               inner xsep=8pt, inner ysep=5pt,
               text width=0.70\textwidth, fill=black!3},
  cert/.style={stage, fill=orange!12, draw=orange!55!black!60},
  tool/.style={font=\scriptsize\itshape, text=black!55, align=left,
               anchor=west},
  arr/.style={->, semithick, black!60}]

\node[stage] (s1) {\textbf{Exact reduction}\; (\S\ref{sec:BB-recurrence})\\[1.5pt]
  $d_{k+1}^i=\dfrac{a_k-\lambda_i}{a_k}\,d_k^i$,
  \qquad
  $(a_k,p_k)=F_\lambda^{k}\bigl(D_\lambda(p_0)\bigr)$\\[1.5pt]
  goal:\ \ $10^{-6}\leq\abs{a_k-\lambda_i}/a_k\leq0.61$ for all $i,k$};

\node[cert, below=of s1] (s2) {\textbf{Certified periodic object}\;
  (\S\ref{sec:certificate}, dimension $4$)\\[1.5pt]
  eight equations $G(u_*)=0$:\\
  seven-cycle $F^{7}(z_*)=z_*$\ $\oplus$\
  admissible landing $F^{15}(D(r_*))=z_*$;\\
  certified $Z_t>0$,\ separations $\abs{a_t-\bar\lambda_i}$,\
  Lyapunov pair $(P,Q_7)$,\ transverse $\tau(\bar\nu)$};

\node[stage, below=of s2] (s3) {\textbf{Attraction}\; (\S\ref{sec:floquet})\\[1.5pt]
  $P\succ0$,\ \ $P-A_7^{T}PA_7\succ0$
  \ $\Longrightarrow$\ $\norm{A_7}_P<1$
  \ $\Longrightarrow$\ the cycle has an open basin};

\node[stage, below=of s3] (s4) {\textbf{Rates and openness in dimension $4$}\;
  (\S\ref{sec:four-open})\\[1.5pt]
  translation $\lambda\mapsto\lambda+4\one$
  \ $\Rightarrow$\ upper factor $0.61$;\\
  certified separation \ $\Rightarrow$\ lower factor $10^{-6}$;\\
  hyperbolicity \ $\Rightarrow$\ open family $L_4\times P_4$};

\node[stage, below=of s4] (s5) {\textbf{Every finite dimension}\;
  (\S\ref{sec:higher-n})\\[1.5pt]
  boundary cycle on an invariant simplex face;\\
  transverse multipliers $\tau(\nu)<0.049$\\
  $\Rightarrow$\ attraction persists\
  $\Rightarrow$\ open $L_n\times P_n$ for each $n\geq4$};

\node[stage, below=of s5] (s6) {\textbf{Lift and measure}\;
  (\S\ref{sec:lift})\\[1.5pt]
  $\Theta_\sigma$,\ projector weights $p(A,g)$,\
  $(A,x)\mapsto(A,Ax-b)$ diffeomorphisms\\
  $\Rightarrow$\ $\Omega_n$ open and nonempty
  \ $\Rightarrow$\ positive measure
  \ $\Rightarrow$\ Theorem~\ref{thm:main}};

\draw[arr] (s1) -- (s2)
  node[tool, right=2.5mm, midway]
  {interval Newton $+$ Banach\ (Lem.~\ref{lem:exact-root})};
\draw[arr] (s2) -- (s3)
  node[tool, right=2.5mm, midway]
  {discrete Lyapunov inequality\ (Lem.~\ref{lem:A-schur},
   \ref{lem:nonlinear-attraction})};
\draw[arr] (s3) -- (s4)
  node[tool, right=2.5mm, midway]
  {translation conjugacy $+$ IFT\ (Lem.~\ref{lem:translation},
   Prop.~\ref{prop:n4-open})};
\draw[arr] (s4) -- (s5)
  node[tool, right=2.5mm, midway]
  {transverse Floquet bound\ (Prop.~\ref{prop:n-open})};
\draw[arr] (s5) -- (s6)
  node[tool, right=2.5mm, midway]
  {three changes of variables\ (\S\ref{sec:lift})};

\end{tikzpicture}
\caption{Map of the proof of Theorem~\ref{thm:main}.  The shaded box is the
computer-assisted step (Certificate~\ref{cert:main}); every other stage is a
pen-and-paper argument that consumes its certified strict inequalities.
Arrow labels name the tool carrying each implication.}
\label{fig:proof-map}
\end{figure}

\section{Exact reduction of BB1 to a projective rational map}
\label{sec:BB-recurrence}

We first derive the scalar recurrence without suppressing the exceptional
index $k=0$.

\begin{lemma}[The BB1 step on a quadratic]\label{lem:bb-step}
Suppose $A\succ0$ and $g_{k-1}\neq0$.  If $x_k=x_{k-1}-\alpha_{k-1}g_{k-1}$
with $\alpha_{k-1}\neq0$, then the BB1 quotient in \eqref{eq:bb1} is positive
and equals
\begin{equation}\label{eq:bb-rayleigh}
 \alpha_k=\frac{g_{k-1}^Tg_{k-1}}{g_{k-1}^TAg_{k-1}}.
\end{equation}
In particular, the initialization \eqref{eq:alpha0} gives $\alpha_1=\alpha_0$.
\end{lemma}

\begin{proof}
Because $g_k=Ax_k-b$ and $g_{k-1}=Ax_{k-1}-b$, subtraction gives
\[
 y_{k-1}=g_k-g_{k-1}=A(x_k-x_{k-1})=As_{k-1}.
\]
The preceding gradient step gives
$s_{k-1}=-\alpha_{k-1}g_{k-1}$.  Therefore
\begin{align*}
 s_{k-1}^Ts_{k-1}
 &=(-\alpha_{k-1}g_{k-1})^T(-\alpha_{k-1}g_{k-1})
 =\alpha_{k-1}^2g_{k-1}^Tg_{k-1},\\
 s_{k-1}^Ty_{k-1}
 &=s_{k-1}^TAs_{k-1}
 =\alpha_{k-1}^2g_{k-1}^TAg_{k-1}.
\end{align*}
Since $A\succ0$ and $g_{k-1}\neq0$, the denominator
$g_{k-1}^TAg_{k-1}$ is strictly positive.  Since
$\alpha_{k-1}^2>0$, cancellation yields \eqref{eq:bb-rayleigh}; both its
numerator and denominator are positive.  Taking $k=1$ and using
\eqref{eq:alpha0} proves $\alpha_1=\alpha_0$.
\end{proof}

Fix an orthogonal diagonalization
\begin{equation}\label{eq:eigendecomp}
 A=Q\Lambda Q^T,
 \qquad \Lambda=\diag(\lambda_1,\ldots,\lambda_n),
 \qquad 0<\lambda_1<\cdots<\lambda_n,
\end{equation}
and define the spectral gradient coordinates
\begin{equation}\label{eq:d-def}
 d_k:=Q^Tg_k,\qquad d_k^i=(Q^Tg_k)_i.
\end{equation}
Because $Q$ is orthogonal,
\begin{equation}\label{eq:rayleigh-coordinates}
 g_k^Tg_k=\sum_{j=1}^n(d_k^j)^2,
 \qquad
 g_k^TAg_k=d_k^T\Lambda d_k
 =\sum_{j=1}^n\lambda_j(d_k^j)^2.
\end{equation}
The gradient update is
\[
 g_{k+1}=A(x_k-\alpha_kg_k)-b
 =(Ax_k-b)-\alpha_kAg_k=(I-\alpha_kA)g_k.
\]
Multiplication by $Q^T$ and use of $Q^TAQ=\Lambda$ give
\begin{equation}\label{eq:d-gradient-step}
 d_{k+1}^i=(1-\alpha_k\lambda_i)d_k^i.
\end{equation}
For $k=0$, substitution of \eqref{eq:alpha0} and
\eqref{eq:rayleigh-coordinates} into \eqref{eq:d-gradient-step} yields
\begin{equation}\label{eq:exceptional-recurrence}
 d_1^i=d_0^i
 \frac{\sum_{j=1}^n(\lambda_j-\lambda_i)(d_0^j)^2}
      {\sum_{j=1}^n\lambda_j(d_0^j)^2}.
\end{equation}
For $k\geq1$, Lemma~\ref{lem:bb-step} instead uses $g_{k-1}$, and hence
\begin{equation}\label{eq:delayed-recurrence}
 d_{k+1}^i=d_k^i
 \frac{\sum_{j=1}^n(\lambda_j-\lambda_i)(d_{k-1}^j)^2}
      {\sum_{j=1}^n\lambda_j(d_{k-1}^j)^2}.
\end{equation}
Equations \eqref{eq:exceptional-recurrence}--\eqref{eq:delayed-recurrence}
are exactly the delayed recurrence certified below.

We now remove the common magnitude of $d_k$.  Whenever $d_k\neq0$, define
\begin{equation}\label{eq:projective-weights}
 x_{k,i}:=(d_k^i)^2,
 \qquad S_k:=\sum_{j=1}^nx_{k,j},
 \qquad p_{k,i}:=\frac{x_{k,i}}{S_k}.
\end{equation}
Then $p_k$ belongs to the simplex
\[
 \Delta_{n-1}:=\left\{p\in\R^n:p_i\geq0,\ \sum_{i=1}^np_i=1\right\}.
\]
For $p\in\Delta_{n-1}$ and $a\in\R$, define
\begin{align}
 \mu_\lambda(p)&:=\sum_{i=1}^n\lambda_ip_i,
 \label{eq:mu}\\
 Z_\lambda(a,p)&:=\sum_{i=1}^np_i(a-\lambda_i)^2,
 \label{eq:Z}\\
 (T_{\lambda,a}(p))_i&:=
 \frac{p_i(a-\lambda_i)^2}{Z_\lambda(a,p)},
 \label{eq:T}\\
 F_\lambda(a,p)&:=\bigl(\mu_\lambda(p),T_{\lambda,a}(p)\bigr),
 \qquad D_\lambda(p):=\bigl(\mu_\lambda(p),p\bigr).
 \label{eq:F-D}
\end{align}
The map $T_{\lambda,a}$ is used only if $Z_\lambda(a,p)>0$.

\begin{lemma}[Exact projective state and multiplier identities]
\label{lem:projective}
Suppose $d_0\neq0$, and let $p_0$ be defined by
\eqref{eq:projective-weights}.  Set
\begin{equation}\label{eq:a-index}
 a_0:=\mu_\lambda(p_0),
 \qquad a_k:=\mu_\lambda(p_{k-1})\quad(k\geq1).
\end{equation}
As long as the recurrence is defined,
\begin{equation}\label{eq:state-iterate}
 (a_k,p_k)=F_\lambda^k(D_\lambda(p_0))\qquad(k\geq0),
\end{equation}
and, for every $i$ and $k\geq0$,
\begin{align}
 d_{k+1}^i&=d_k^i\frac{a_k-\lambda_i}{a_k},
 \label{eq:one-multiplier}\\
 \abs{d_k^i}&=\abs{d_0^i}
 \prod_{t=0}^{k-1}\frac{\abs{a_t-\lambda_i}}{a_t}.
 \label{eq:product-multiplier}
\end{align}
The empty product for $k=0$ is one.  In particular,
$a_1=a_0=\mu_\lambda(p_0)$.
\end{lemma}

\begin{proof}
For $k\geq1$, the scalar quotient in \eqref{eq:delayed-recurrence} is
\begin{align*}
 \frac{\sum_j(\lambda_j-\lambda_i)(d_{k-1}^j)^2}
      {\sum_j\lambda_j(d_{k-1}^j)^2}
 &=\frac{\sum_j\lambda_jx_{k-1,j}
          -\lambda_i\sum_jx_{k-1,j}}
        {\sum_j\lambda_jx_{k-1,j}}\\
 &=\frac{S_{k-1}\sum_j\lambda_jp_{k-1,j}
          -\lambda_iS_{k-1}}
        {S_{k-1}\sum_j\lambda_jp_{k-1,j}}\\
 &=\frac{\mu_\lambda(p_{k-1})-\lambda_i}
        {\mu_\lambda(p_{k-1})}
 =\frac{a_k-\lambda_i}{a_k}.
\end{align*}
The same computation applied to \eqref{eq:exceptional-recurrence}, with
$p_0$ in place of $p_{k-1}$, gives the formula for $k=0$.  This proves
\eqref{eq:one-multiplier}.

Squaring \eqref{eq:one-multiplier} gives
\begin{equation}\label{eq:x-update}
 x_{k+1,i}=x_{k,i}\frac{(a_k-\lambda_i)^2}{a_k^2}.
\end{equation}
Summing \eqref{eq:x-update} over $i$ and using $x_{k,i}=S_kp_{k,i}$ yields
\begin{align*}
 S_{k+1}
 &=\sum_i x_{k,i}\frac{(a_k-\lambda_i)^2}{a_k^2}
 =\frac{S_k}{a_k^2}\sum_ip_{k,i}(a_k-\lambda_i)^2
 =\frac{S_k}{a_k^2}Z_\lambda(a_k,p_k).
\end{align*}
If $Z_\lambda(a_k,p_k)>0$, division of \eqref{eq:x-update} by this identity
gives
\[
 p_{k+1,i}
 =\frac{p_{k,i}(a_k-\lambda_i)^2}{Z_\lambda(a_k,p_k)}
 =(T_{\lambda,a_k}(p_k))_i.
\]
Definition \eqref{eq:a-index} also gives
$a_{k+1}=\mu_\lambda(p_k)$.  Hence
\[
 F_\lambda(a_k,p_k)=(a_{k+1},p_{k+1}).
\]
Since $(a_0,p_0)=D_\lambda(p_0)$, induction on $k$ proves
\eqref{eq:state-iterate}.  Finally, taking absolute values in
\eqref{eq:one-multiplier} and multiplying the identities for
$t=0,\ldots,k-1$ proves \eqref{eq:product-multiplier}.

Because all $\lambda_i>0$ and $p$ is a probability vector,
$\mu_\lambda(p)\geq\lambda_1>0$.  Thus $a_k$ never causes division by zero.
The only remaining possible singularity is $Z_\lambda(a_k,p_k)=0$; the
construction below keeps every orbit in a region where $Z$ is strictly
positive.
\end{proof}

\section{The certified four-dimensional periodic object}
\label{sec:certificate}

The only computer-assisted part of the proof is the validation of one exact
zero of eight rational equations and of several strict inequalities at that
zero. The validation is performed by the supplied script
\texttt{interval\_cert.py}, provided in Appendix \ref{app:interval-certificate}.
The script uses only Python's standard \texttt{decimal} module and contains
all entries of the rational preconditioner and Lyapunov matrices.

\subsection{The eight equations}

Begin with the unshifted four-point spectrum
\begin{equation}\label{eq:unshifted-spectrum}
 \bar\lambda=
 \bigl(1,\ 1.8786699041860466,\ \bar\lambda_3,\ 4\bigr).
\end{equation}
The displayed second eigenvalue is a terminating decimal and is interpreted as
the exact corresponding rational number.  Use the direct state chart
\[
 z=(a,p_1,p_2,p_3),\qquad p_4=1-p_1-p_2-p_3,
\]
and the initialization chart
\[
 r=(r_1,r_2,r_3),\qquad r_4=1-r_1-r_2-r_3.
\]
The unknown is
\[
 u=(\bar\lambda_3,a,p_1,p_2,p_3,r_1,r_2,r_3)\in\R^8.
\]
Define, wherever the iterates are well-defined,
\begin{equation}\label{eq:G}
 G(u):=
 \begin{pmatrix}
  F_{\bar\lambda}^{7}(z)-z\\[2mm]
  F_{\bar\lambda}^{15}(D_{\bar\lambda}(r))-z
 \end{pmatrix}\in\R^8.
\end{equation}
The first four equations require a point fixed by the seven-step map.  The
last four require an exact fifteen-step connection from an admissible BB initial state $D_{\bar\lambda}(r)$.  The landing condition is essential: a periodic state in the four-dimensional $(a,p)$ space need not lie in the forward orbit of the three-dimensional admissible-initialization hypersurface $a=\mu(p)$.

Let $c=(c_1,\ldots,c_8)$ be the exact terminating-decimal vector
\begingroup\small
\begin{align*}
c_1={}&2.71278149106533722827430219171662596122131200973311943581550524557629575811,\\
c_2={}&1.33344967252388129918980034968384933800281196626470897628663867165111253656,\\
c_3={}&0.828108598362557830284706279933642675447754907755381063536594022726890227454,\\
c_4={}&0.167043051370226631389809462759502319781038286313814177038878187525592790929,\\
c_5={}&0.004717957350367412952812533939501994168806798103523267952088050658941184644,\\
c_6={}&0.0000134113750442893933116720965327494222816004773939806475307653666721805974245,\\
c_7={}&0.989171705769078149047269711515135535847128223427916068474163743526926233731,\\
c_8={}&0.00000349260028439145589133467903469364875523517774851271154130219629904643379924.
\end{align*}
\endgroup
Set
\begin{equation}\label{eq:X-box}
 X:=c+[-10^{-55},10^{-55}]^8.
\end{equation}
At the center, and uniformly throughout $X$, the eliminated probabilities
satisfy
\begin{align}
 p_4&\in
 [0.0001303929168481253726717233673530106024000078272814911,\notag\\[-1mm]
 &\hspace{35mm}0.0001303929168481253726717233673530106024000078272814918],
 \label{eq:p4-box}\\
 r_4&\in
 [0.0108113902555931701035272817092970210818349409169414378,\notag\\[-1mm]
 &\hspace{35mm}0.0108113902555931701035272817092970210818349409169414385].
 \label{eq:r4-box}
\end{align}
Thus $p$ and $r$ are in the relative interior of $\Delta_3$, and
$1<1.8786699041860466<\bar\lambda_3<4$ throughout $X$.

\subsection{What the interval calculation proves}

\begin{certificate}[Validated root, stability, and separation]
\label{cert:main}
With $G$, $c$, and $X$ defined above, the supplied directed-rounding
calculation proves all of the following statements.

\begin{enumerate}[label=\textup{(C\arabic*)}]
\item\label{cert:K}
There is an explicitly recorded exact rational matrix $R\in\R^{8\times8}$
such that, with $H(u)=u-RG(u)$,
\begin{equation}\label{eq:K-inclusion}
 H(X)\subset
 c-RG(c)+\bigl(I-R[DG(X)]\bigr)(X-c)
 \subset\intt X.
\end{equation}
The largest coordinate ratio of the final box radius to $10^{-55}$ is less
than
\[
 3.876379\times10^{-20}.
\]

\item\label{cert:contraction}
The same derivative enclosure satisfies
\begin{equation}\label{eq:H-contraction-cert}
 \sup_{u\in X}\norm{DH(u)}_\infty
 \leq \norm{I-R[DG(X)]}_\infty
 <7.200691\times10^{-27}<1.
\end{equation}

\item\label{cert:denominators}
Every normalization denominator occurring in the seven cycle steps and the
fifteen connection steps is positive.  More explicitly, the same outward
evaluation gives
\begin{equation}\label{eq:Z-bounds}
 \min_{0\leq t<7}Z_t>0.1516357118221540,
 \qquad
 \min_{0\leq t<15}Z_t^{\rm conn}>0.0009510722601452239.
\end{equation}

\item\label{cert:separation}
For every exact $u\in X$, the finite connection and cycle obey the uniform
separation bounds
\begin{align}
 \min_{\substack{0\leq t\leq15\\1\leq i\leq4}}
 \abs{a_t-\bar\lambda_i}
 &>1.3886903675\times10^{-4},
 \label{eq:transient-separation}\\
 \min_{\substack{0\leq t<7\\1\leq i\leq4}}
 \abs{a_t-\bar\lambda_i}
 &>7.5691389311\times10^{-2}.
 \label{eq:cycle-separation}
\end{align}
Both statements remain true, with the same displayed lower bounds, when an
additional test eigenvalue $\bar\nu\in[3.49,3.51]$ is included in the minimum.

\item\label{cert:lyapunov}
Let $A_7=D_zF_{\bar\lambda}^{7}(z)$, with $z$ ranging over the $z$-part of
$X$.  The script records an exact symmetric rational matrix $P$ and proves,
for the exact root obtained below,
\begin{equation}\label{eq:lyapunov-ineq}
 P\succ0,
 \qquad Q_7:=P-A_7^TPA_7\succ0.
\end{equation}
The outward-rounded $LDL^T$ pivots for $P$ have the positive lower bounds
\begin{equation}\label{eq:P-pivots}
 977656.6634223271,\quad
 65683881.36341604,\quad
 9.217413243679623,\quad
 323.9590422006289,
\end{equation}
and the row diagonal-dominance margins for $Q_7$ have the lower bounds
\begin{equation}\label{eq:Q-margins}
 \begin{split}
 0.9999999999999999999976374,\quad
 0.9999999999999999994270558,\\
 0.9999999999999999994222700,\quad
 0.9999999999999999993948258.
 \end{split}
\end{equation}

\item\label{cert:transverse}
For every $\bar\nu\in[3.49,3.51]$, the seven-step transverse multiplier
defined in \eqref{eq:tau} below satisfies
\begin{equation}\label{eq:tau-cert}
 0.0230006289035345
 <\tau(\bar\nu)
 <0.0482036188628888<1.
\end{equation}

\item\label{cert:period}
For the first two phases of the cycle,
\begin{equation}\label{eq:phase-difference}
 a_0-a_1>0.1782019618055700634.
\end{equation}
\end{enumerate}
\end{certificate}

These machine statements are rigorous real inequalities by the standard
directed-rounding validation principle
\citep{krawczyk1969newton,rump2010verification}.  A decimal literal is first
converted to a \texttt{Decimal} exactly.  For each addition, subtraction,
multiplication, and reciprocal, the script evaluates once with rounding
toward $-\infty$ and once with rounding toward $+\infty$; if the input
intervals contain the exact real inputs, the output interval therefore
contains the exact result of that elementary operation.  Structural induction
over an arithmetic expression then shows that its interval evaluation
encloses its entire real range over the input box.  The
automatic-differentiation class stores a pair $(v,d)$ of an interval value
and a vector of interval partial derivatives and propagates them by the exact
identities $D(u+v)=Du+Dv$, $D(uv)=vDu+uDv$, $D(u^{-1})=-u^{-2}Du$; the same
induction, applied simultaneously to values and derivatives, shows that the
computed matrix $[DG(X)]$ contains $DG(u)$ entrywise for every $u\in X$.
Each reciprocal operation asserts that its denominator interval excludes
zero, and the explicit positive bounds \eqref{eq:Z-bounds} strengthen this
check.  All magnitudes used are far from the decimal exponent limits, so no
overflow or underflow is involved.

The centered enclosure in \eqref{eq:K-inclusion} follows from the identity
\begin{align*}
 H(u)-c
 &=H(c)-c+\int_0^1DH(c+t(u-c))(u-c)\,dt\\
 &=-RG(c)+\int_0^1\bigl(I-RDG(c+t(u-c))\bigr)(u-c)\,dt.
\end{align*}
valid because $X$ is convex and $H$ is $C^1$ on a neighborhood of $X$.
Every derivative matrix in the integral belongs entrywise to $I-R[DG(X)]$,
and interval matrix--box multiplication encloses every product with
$u-c\in X-c$ as well as its average over $t\in[0,1]$.  The second inclusion
in \eqref{eq:K-inclusion} is the directed-rounding output.

We next turn the interval statements into exact analytic conclusions.

\begin{figure}[t]
\centering
\includegraphics[width=\textwidth]{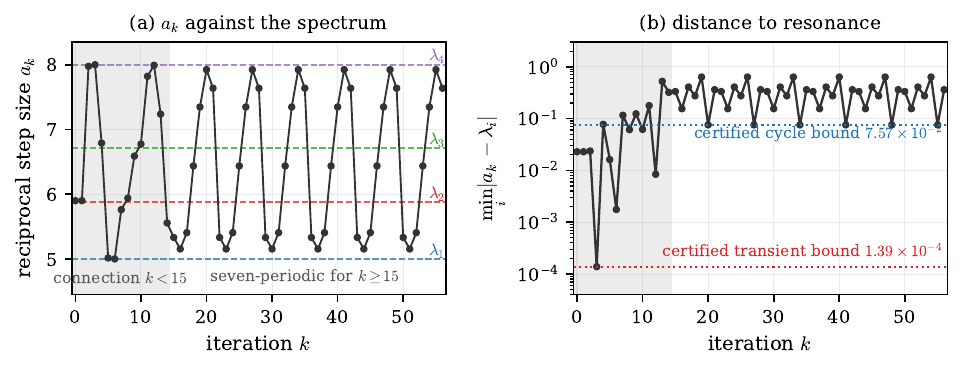}
\caption{The certified projective orbit at the central parameter.
\emph{(a)} The reciprocal step size $a_k$ against the eigenvalues: the
admissible initialization lands on the seven-cycle after exactly fifteen
steps and then repeats.  \emph{(b)} The distance from $a_k$ to the nearest
eigenvalue.  The certified separation bounds of
Certificate~\ref{cert:main}\ref{cert:separation} are sharp for this orbit:
the transient minimum at $k=3$ and the cycle minimum match the certified
lower bounds to the displayed digits.}
\label{fig:certified-orbit}
\end{figure}

\begin{lemma}[Existence and uniqueness of an exact solution]
\label{lem:exact-root}
There is a unique $u_*\in X$ satisfying $G(u_*)=0$.
\end{lemma}

\begin{proof}
The box $X$ is a nonempty closed subset of $(\R^8,\norm{\cdot}_\infty)$ and
hence a complete metric space.
Certificate~\ref{cert:main}\ref{cert:K} gives $H(X)\subset\intt X\subset X$,
so $H$ maps $X$ into itself, and by
Certificate~\ref{cert:main}\ref{cert:denominators} every denominator is
nonzero on an open neighborhood of $X$, so $H$ is $C^1$ there.  For
$u,v\in X$ the segment $[v,u]$ lies in the convex set $X$, so the mean-value
inequality together with Certificate~\ref{cert:main}\ref{cert:contraction}
gives
\[
 \norm{H(u)-H(v)}_\infty
 \leq(7.200691\times10^{-27})\norm{u-v}_\infty.
\]
Banach's fixed-point theorem yields a unique $u_*\in X$ with $H(u_*)=u_*$.

The bound \eqref{eq:H-contraction-cert} also forces $R$ to be invertible: if
$y^TR=0$ for some $y\neq0$, then $y^T(I-RDG(u))=y^T$, so $1$ would be an
eigenvalue of $(I-RDG(u))^T$, contradicting
$\rhoSp(I-RDG(u))\leq\norm{I-RDG(u)}_\infty<1$.  Hence $H(u_*)=u_*$ reads
$RG(u_*)=0$ and gives $G(u_*)=0$.  Conversely, every zero of $G$ in $X$ is a
fixed point of $H$, so $u_*$ is the only such zero.
\end{proof}

Write the exact components as
\begin{equation}\label{eq:exact-components}
 u_*=(\bar\lambda_3^*,z_*,r_*),
 \qquad z_*=(a_*,p_*).
\end{equation}
Unpacking $G(u_*)=0$ gives the exact, rather than numerical, identities
\begin{equation}\label{eq:cycle-landing}
 F_{\bar\lambda}^{7}(z_*)=z_*,
 \qquad
 F_{\bar\lambda}^{15}(D_{\bar\lambda}(r_*))=z_*.
\end{equation}
The third eigenvalue lies in the radius-$10^{-55}$ interval centered at
$c_1$, so in particular
\[
 \bar\lambda_3^*=2.71278149106533722827430219171662596122\ldots.
\]
Equation \eqref{eq:phase-difference} shows $F_{\bar\lambda}(z_*)\neq z_*$.  The least period of $z_*$ divides $7$; since $7$ is prime and the period is not $1$, it equals $7$.

\section{Floquet stability and nonlinear attraction}
\label{sec:floquet}

Let
\begin{equation}\label{eq:A7}
 A_7:=D_zF_{\bar\lambda}^{7}(z_*).
\end{equation}
The derivative exists because all seven denominators are positive. We now derive every stability implication of Certificate \ref{cert:main}\ref{cert:lyapunov}.

\begin{lemma}[The certified matrices are positive definite]
\label{lem:PQ-positive}
The exact matrices $P$ and $Q_7=P-A_7^TPA_7$ satisfy $P\succ0$ and
$Q_7\succ0$.
\end{lemma}

\begin{proof}
The interval $LDL^T$ recursion is performed without pivoting; every division
is valid because each preceding pivot interval is positive.  It encloses an
exact factorization $P=LDL^T$ with $L$ unit lower triangular and the four
diagonal entries of $D$ bounded below by the positive numbers in
\eqref{eq:P-pivots}; hence $P\succ0$.  The matrix $Q_7$ is symmetric, and
the interval calculation proves strict row diagonal dominance,
\begin{equation}\label{eq:strict-dd}
 (Q_7)_{ii}-\sum_{j\neq i}\abs{(Q_7)_{ij}}\geq m_i>0,
\end{equation}
with the margins $m_i$ bounded below by \eqref{eq:Q-margins}; in particular
every diagonal entry is positive.  By Gershgorin's circle theorem, every
eigenvalue of $Q_7$ is positive, so $Q_7\succ0$.
\end{proof}

\begin{lemma}[Schur stability of the seven-step derivative]
\label{lem:A-schur}
Every complex eigenvalue of $A_7$ has modulus strictly smaller than one.
\end{lemma}

\begin{proof}
Let $v\in\mathbb C^4\setminus\{0\}$ and $\zeta\in\mathbb C$ satisfy
$A_7v=\zeta v$.  A real symmetric positive-definite matrix defines a positive
Hermitian form on $\mathbb C^4$, so $v^*Pv>0$ and $v^*Q_7v>0$.  Using that
$A_7$ and $P$ are real,
\[
 v^*Q_7v
 =v^*Pv-(A_7v)^*P(A_7v)
 =(1-\abs{\zeta}^2)v^*Pv.
\]
Both sides are strictly positive, so $\abs{\zeta}<1$.
\end{proof}

\begin{lemma}[Nonlinear attraction in an adapted norm]
\label{lem:nonlinear-attraction}
The point $z_*$ is a locally attracting fixed point of
$\Phi:=F_{\bar\lambda}^{7}$.
\end{lemma}

\begin{proof}
Define $\norm{x}_P=(x^TPx)^{1/2}$.  Since $P\succ0$, this is a norm.  On the
compact $P$-unit sphere $S_P=\{x:x^TPx=1\}$ the continuous function
$x\mapsto x^TQ_7x$ is strictly positive by Lemma~\ref{lem:PQ-positive} and
therefore has a positive minimum $\delta>0$; by homogeneity,
\begin{equation}\label{eq:Q-delta-P}
 x^TQ_7x\geq\delta x^TPx\qquad\text{for every }x\in\R^4.
\end{equation}
Since $Q_7=P-A_7^TPA_7$, \eqref{eq:Q-delta-P} gives
\[
 \norm{A_7x}_P^2
 =x^T(P-Q_7)x
 \leq(1-\delta)\norm{x}_P^2,
\]
and the left side is nonnegative, so $0<\delta\leq1$ and
\begin{equation}\label{eq:A-P-contract}
 \norm{A_7}_P\leq q_0:=\sqrt{1-\delta}<1.
\end{equation}

All normalization denominators at the cycle phases are positive by
\eqref{eq:Z-bounds} and remain positive on a small Euclidean neighborhood of
each phase, so $\Phi$ is $C^1$ near $z_*$ with continuous derivative and
$D\Phi(z_*)=A_7$.  Choose any $q$ with $q_0<q<1$.  Continuity of $D\Phi$ in
the induced $P$-norm gives $\varepsilon>0$ such that
\begin{equation}\label{eq:Dphi-local}
 \norm{D\Phi(z)}_P\leq q
 \quad\text{whenever}\quad \norm{z-z_*}_P\leq\varepsilon.
\end{equation}
The closed $P$-ball in \eqref{eq:Dphi-local} is convex, so for any $z$ in
this ball the mean-value inequality gives
\[
 \norm{\Phi(z)-z_*}_P
 =\norm{\Phi(z)-\Phi(z_*)}_P
 \leq q\norm{z-z_*}_P.
\]
Thus the ball is forward invariant under $\Phi$, and induction gives
$\norm{\Phi^m(z)-z_*}_P\leq q^m\norm{z-z_*}_P\to0$.
\end{proof}

\section{Translation and an open four-dimensional family}
\label{sec:four-open}

The projective dynamics depends only on differences $a-\lambda_i$, whereas
the original BB multiplier also contains the denominator $a$.  A spectral
translation exploits this distinction.

\begin{lemma}[Translation conjugacy]\label{lem:translation}
For $s\in\R$, let $\lambda'=\lambda+s\one$ and
$C_s(a,p)=(a+s,p)$.  Wherever the maps are defined,
\begin{equation}\label{eq:translation-conjugacy}
 F_{\lambda'}\circ C_s=C_s\circ F_\lambda,
 \qquad D_{\lambda'}=C_s\circ D_\lambda.
\end{equation}
\end{lemma}

\begin{proof}
Because $\sum_ip_i=1$,
\[
 \mu_{\lambda'}(p)=\sum_i(\lambda_i+s)p_i
 =\mu_\lambda(p)+s.
\]
Also $(a+s)-(\lambda_i+s)=a-\lambda_i$, so
\[
 Z_{\lambda'}(a+s,p)=\sum_ip_i((a+s)-(\lambda_i+s))^2
 =Z_\lambda(a,p),
\]
and every normalized weight in \eqref{eq:T} is unchanged.  Substitution into
\eqref{eq:F-D} proves both identities.
\end{proof}

Take $s=4$.  The shifted central spectrum is
\begin{equation}\label{eq:shifted-spectrum}
 \lambda^*=
 \bigl(5,\ 5.8786699041860466,\ 6.712781491065337228\ldots,\ 8\bigr).
\end{equation}
Translation leaves the projective weights, all differences, all $Z$ values,
and all state derivatives unchanged; it adds four only to the scalar $a$.
Let $z_*^+=C_4(z_*)$.  Then
\begin{equation}\label{eq:shifted-cycle-connection}
 F_{\lambda^*}^{7}(z_*^+)=z_*^+,
 \qquad
 F_{\lambda^*}^{15}(D_{\lambda^*}(r_*))=z_*^+.
\end{equation}

\begin{proposition}[Uniform factor bounds in dimension four]
\label{prop:n4-open}
There are nonempty open neighborhoods
\[
 L_4\subset\{\lambda\in\R^4:0<\lambda_1<\cdots<\lambda_4\}
 \quad\text{and}\quad
 P_4\subset\intt\Delta_3
\]
of $\lambda^*$ and $r_*$, respectively, such that for every
$(\lambda,p_0)\in L_4\times P_4$, the orbit
$(a_k,p_k)=F_\lambda^k(D_\lambda(p_0))$ is defined for all $k$ and
\begin{equation}\label{eq:factor-n4}
 10^{-6}\leq\frac{\abs{a_k-\lambda_i}}{a_k}\leq0.61
 \qquad(1\leq i\leq4,\ k\geq0).
\end{equation}
\end{proposition}

\begin{proof}
We divide the argument into continuation, uniform attraction, finite-time
entry, and the numerical factor bounds.

\smallskip
\noindent\emph{Step 1: continuation of the seven-cycle.}
Define $\Psi(\lambda,z):=F_\lambda^{7}(z)-z$.  All seven $Z$ denominators are
positive at the cycle by \eqref{eq:Z-bounds} and, by continuity, positive
nearby, so $\Psi$ is $C^1$ near $(\lambda^*,z_*^+)$.  Identity
\eqref{eq:shifted-cycle-connection} gives $\Psi(\lambda^*,z_*^+)=0$, and
\[
 D_z\Psi(\lambda^*,z_*^+)=A_7-I
\]
is invertible because $1$ is not an eigenvalue of $A_7$
(Lemma~\ref{lem:A-schur}).  The implicit-function theorem gives an open
neighborhood $U_\lambda$ of $\lambda^*$ and a $C^1$ periodic-point function
$z(\lambda)$ satisfying
\begin{equation}\label{eq:continued-cycle}
 F_\lambda^{7}(z(\lambda))=z(\lambda),
 \qquad z(\lambda^*)=z_*^+.
\end{equation}

\smallskip
\noindent\emph{Step 2: uniform local attraction.}
At the central point, $P-A_7^TPA_7\succ0$.  Positive definiteness is open in the entries: if $Q_7\succ0$, its minimum on the Euclidean unit sphere is some $m>0$, and every symmetric perturbation $E$ with $\norm{E}_2<m/2$ satisfies $x^T(Q_7+E)x\geq m/2$ for every Euclidean unit $x$. The matrix $D_zF_\lambda^{7}(z(\lambda))$ depends continuously on $\lambda$, so after shrinking $U_\lambda$, the strict inequality
\[
 P-D_zF_\lambda^{7}(z(\lambda))^T
 P D_zF_\lambda^{7}(z(\lambda))\succ0
\]
holds for every $\lambda\in U_\lambda$.  Repeating the compact-unit-sphere
argument from Lemma~\ref{lem:nonlinear-attraction}, now over the closure of a
sufficiently small parameter ball, gives a common contraction constant
$q<1$ and a common radius $\varepsilon>0$ such that the $P$-ball
\begin{equation}\label{eq:parameter-trap}
 B_\lambda:=\{z:\norm{z-z(\lambda)}_P<\varepsilon\}
\end{equation}
is mapped into itself by $F_\lambda^7$ for every parameter in that smaller
ball.  Shrinking $\varepsilon$ if necessary, all seven intermediate images
$F_\lambda^j(B_\lambda)$, $0\leq j<7$, stay in neighborhoods where $Z>0$ and
where every cycle-phase separation remains strict.

\smallskip
\noindent\emph{Step 3: admissible initial conditions enter the trapping tube.}
At the central parameter, \eqref{eq:shifted-cycle-connection} gives exact
entry after fifteen steps.  The map
\begin{equation}\label{eq:finite-entry-map}
 (\lambda,r)\longmapsto F_\lambda^{15}(D_\lambda(r))
\end{equation}
is continuous near $(\lambda^*,r_*)$: it is a finite composition of rational
maps, and every denominator on the central connection is positive by
\eqref{eq:Z-bounds}.  Since $B_{\lambda^*}$ is an open neighborhood of
$z_*^+$, continuity gives a joint open neighborhood $W$ of
$(\lambda^*,r_*)$ whose image under \eqref{eq:finite-entry-map} lies in the
corresponding parameter-dependent trapping ball; $W$ contains a product
$L_4\times P_4$ of open neighborhoods, and \eqref{eq:p4-box} and
\eqref{eq:r4-box} allow $P_4$ to be chosen inside $\intt\Delta_3$.

For any $(\lambda,p_0)$ in this product, the first fifteen iterates are
defined by finite-time continuity, the fifteenth state lies in $B_\lambda$,
and thereafter the seven-step iterates remain in $B_\lambda$ with the
intermediate states in its seven-phase trapping tube. Hence the orbit is
defined for all time.

\smallskip
\noindent\emph{Step 4: lower factor bound.}
At the central parameter and initial weight, the finite connection has
absolute separation greater than $1.3886903675\times10^{-4}$ by
\eqref{eq:transient-separation}; the cycle has the much larger separation
\eqref{eq:cycle-separation}.  All relevant functions are continuous, so we
may shrink the trapping tube and $L_4\times P_4$ so that
\begin{equation}\label{eq:gap-1e-5}
 \abs{a_k-\lambda_i}>10^{-5}
 \quad\text{for all }i\text{ and all }k\geq0.
\end{equation}
This is a finite continuity assertion for $0\leq k\leq15$ and a uniform
seven-phase assertion in the forward-invariant trapping tube for $k\geq15$.

Shrink $L_4$ further so that $\lambda_4<8.01$.  Every scalar $a_k$ along an
admissible orbit is a convex combination of the eigenvalues, by
$a_0=\mu_\lambda(p_0)$ and $a_k=\mu_\lambda(p_{k-1})$ for $k\geq1$; thus
$a_k\leq\lambda_4<8.01$, and with \eqref{eq:gap-1e-5},
\[
 \frac{\abs{a_k-\lambda_i}}{a_k}
 >\frac{10^{-5}}{8.01}>10^{-6}.
\]

\smallskip
\noindent\emph{Step 5: upper factor bound.}
Shrink $L_4$ once more so that
\begin{equation}\label{eq:spectrum-neighborhood-bounds}
 \lambda_1>4.99,
 \qquad\lambda_4<8.01,
 \qquad\lambda_4-\lambda_1<3.02.
\end{equation}
For every convex combination $a_k\in[\lambda_1,\lambda_4]$ and every
$\lambda_i$ in the same interval,
\[
 \abs{a_k-\lambda_i}\leq\lambda_4-\lambda_1<3.02,
 \qquad a_k\geq\lambda_1>4.99,
\]
so
\[
 \frac{\abs{a_k-\lambda_i}}{a_k}
 <\frac{3.02}{4.99}<0.61.
\]
This completes all parts of \eqref{eq:factor-n4}.
\end{proof}

\section{Extension to every finite dimension}
\label{sec:higher-n}

Fix $n>4$.  We embed the four active eigenvalues at indices
$1,2,3,n$ and choose any distinct central extra eigenvalues
\begin{equation}\label{eq:extra-eigs}
 7.49<\nu_4<\nu_5<\cdots<\nu_{n-1}<7.51.
\end{equation}
Such a choice exists for every finite $n$: for example, equally spaced points in the interval $(7.49,7.51)$ suffice.  The resulting central ordered
spectrum is
\begin{equation}\label{eq:n-spectrum}
 \lambda^{*,n}=
 (5,\ 5.8786699041860466,\ 6.712781491065337228\ldots,
 \nu_4,\ldots,\nu_{n-1},\ 8).
\end{equation}

Give every extra coordinate zero projective weight.  The corresponding face
of $\Delta_{n-1}$ is invariant because each update has the form
\[
 p_i'=\frac{p_i(a-\lambda_i)^2}{Z(a,p)};
\]
if $p_i=0$, then $p_i'=0$.  Thus the four-dimensional seven-cycle persists as
a boundary cycle of the $n$-dimensional projective map.

Use independent local coordinates
\[
 (a,p_1,p_2,p_3,q_4,\ldots,q_{n-1}),
\]
where $q_\ell$ is the weight at $\nu_\ell$, and eliminate the active weight at
the eigenvalue $8$:
\[
 p_n=1-p_1-p_2-p_3-\sum_{\ell=4}^{n-1}q_\ell.
\]
At a boundary-cycle phase, $q=0$ and
\[
 Z_t=\sum_{j\in\{1,2,3,n\}}p_{t,j}(a_t-\lambda_j)^2>0.
\]
For an extra weight,
\begin{equation}\label{eq:q-update}
 q_\ell'=\frac{q_\ell(a-\nu_\ell)^2}{Z(a,p,q)}.
\end{equation}
Differentiate \eqref{eq:q-update} at $q=0$ by the quotient rule gives
\begin{align*}
 \frac{\partial q_\ell'}{\partial q_\ell}\bigg|_{q=0}
 &=\frac{(a-\nu_\ell)^2Z
 -q_\ell(a-\nu_\ell)^2(\partial Z/\partial q_\ell)}{Z^2}
 \bigg|_{q=0}
 =\frac{(a-\nu_\ell)^2}{Z},\\
 \frac{\partial q_\ell'}{\partial q_m}\bigg|_{q=0}
 &=-\frac{q_\ell(a-\nu_\ell)^2(\partial Z/\partial q_m)}{Z^2}
 \bigg|_{q=0}=0\qquad(m\neq\ell).
\end{align*}
Every derivative of $q_\ell'$ with respect to a base variable $a,p_1,p_2,p_3$
also contains the prefactor $q_\ell$ and is zero at $q=0$.  Thus the one-step
derivative at phase $t$ is block upper triangular:
\begin{equation}\label{eq:one-step-block}
 \begin{pmatrix}
  B_t&C_t\\
  0&\diag(\gamma_t(\nu_4),\ldots,\gamma_t(\nu_{n-1}))
 \end{pmatrix},
 \qquad
 \gamma_t(\nu)=\frac{(a_t-\nu)^2}{Z_t}.
\end{equation}
Products of block upper-triangular matrices are block upper triangular, and
their diagonal blocks are the products of the diagonal blocks. Therefore the full seven-step derivative $M_n$ has the form
\begin{equation}\label{eq:tau}
 M_n=
 \begin{pmatrix}
   A_7&*\\
   0&\diag(\tau(\nu_4),\ldots,\tau(\nu_{n-1}))
 \end{pmatrix},
 \qquad
 \tau(\nu):=\prod_{t=0}^{6}\frac{(a_t-\nu)^2}{Z_t}.
\end{equation}
The certificate evaluates the unshifted interval $\bar\nu\in[3.49,3.51]$.  Translation by four leaves every difference and $Z_t$ unchanged, so \eqref{eq:tau-cert} holds for every $\nu\in[7.49,7.51]$, including every extra eigenvalue in \eqref{eq:extra-eigs}.

The characteristic polynomial of a block upper-triangular matrix is the
product of the characteristic polynomials of its diagonal blocks. Thus, the
eigenvalues of $M_n$ are those of $A_7$ together with the scalars
$\tau(\nu_\ell)$.  Lemma~\ref{lem:A-schur} and \eqref{eq:tau-cert} imply
\begin{equation}\label{eq:Mn-schur}
 \rhoSp(M_n)<1.
\end{equation}

Since $\rhoSp(M_n)<1$, we have $\norm{M_n^k}_2\leq C\gamma^k$ for any
$\gamma\in(\rhoSp(M_n),1)$ and some finite $C$, so the series
\begin{equation}\label{eq:Pn-series}
 \widetilde P_n:=\sum_{k=0}^{\infty}(M_n^T)^kM_n^k
\end{equation}
converges, satisfies $\widetilde P_n\succeq I\succ0$, and obeys the exact
identity $\widetilde P_n-M_n^T\widetilde P_nM_n=I$.  It therefore supplies an adapted Lyapunov norm in the full state space, even though the four-dimensional matrix $P$ does not include the new coordinates, and thus the nonlinear attraction proof of Lemma~\ref{lem:nonlinear-attraction} applies verbatim in the full state space.

The boundary cycle lies on a simplex face, but the rational map is smooth on an ordinary ambient Euclidean neighborhood because every cycle denominator is positive, so its attracting neighborhood intersects the simplex in a relative neighborhood containing interior points. To see this explicitly, start from the boundary initialization $r_*$, assign sufficiently small positive values to all extra weights, and subtract their sum from the active weight $r_{*,4}$ at eigenvalue $8$, which is possible because $r_{*,4}>0.0108$ by \eqref{eq:r4-box}. The resulting weight lies in $\intt\Delta_{n-1}$ and can be arbitrarily close to the boundary weight.  The boundary admissible state lands on the cycle after fifteen steps; by continuity of the finite iterate, every sufficiently close interior weight lands in the attracting neighborhood. Choose one such interior weight and call it $r_*^{(n)}$.

We must also continue the cycle under perturbations of all $n$ eigenvalues. The map $\Psi_n(\lambda,z)=F_\lambda^7(z)-z$ is $C^1$ near the central boundary cycle because the denominators are positive, and its state derivative $M_n-I$ is invertible by \eqref{eq:Mn-schur}.  The implicit-function theorem therefore continues the periodic point for every sufficiently small perturbation of all $n$ eigenvalues, and the strict Lyapunov inequality, the finite-time entry property, and all separation inequalities persist after the parameter and initial neighborhoods are shrunk.

The certificate's separation test includes every $\bar\nu\in[3.49,3.51]$, hence after translation every $\nu\in[7.49,7.51]$. The factor-bound calculation in Steps 4--5 of Proposition~\ref{prop:n4-open} therefore applies to active and extra coordinates alike. We have proved the following.

\begin{proposition}[Uniform factor bounds for every finite $n$]
\label{prop:n-open}
For every finite $n\geq4$, there are nonempty open sets
\begin{equation}\label{eq:LnPn}
 L_n\subset\{\lambda\in\R^n:0<\lambda_1<\cdots<\lambda_n\},
 \qquad P_n\subset\intt\Delta_{n-1},
\end{equation}
such that every admissible projective orbit initialized by
$(\lambda,p_0)\in L_n\times P_n$ is defined for all time and satisfies
\begin{equation}\label{eq:factor-all-n}
 10^{-6}\leq\frac{\abs{a_k-\lambda_i}}{a_k}\leq0.61
 \qquad(1\leq i\leq n,\ k\geq0).
\end{equation}
\end{proposition}

There is no hidden uniformity claim as $n\to\infty$: for each fixed finite
$n$, finitely many distinct extra eigenvalues are selected and then a
possibly $n$-dependent open neighborhood is taken.  The rate constants
$10^{-6}$ and $0.61$, however, are the same for every finite $n$.

\section{Lifting the projective construction to BB problems and measure}
\label{sec:lift}

We first obtain componentwise gradient bounds for a diagonal quadratic.  Fix
$\lambda\in L_n$ and choose any $p_0\in P_n$, any radius $r>0$, and any sign
vector $\sigma\in\{\pm1\}^n$.  Set
\begin{equation}\label{eq:d-from-p}
 d_0^i:=\sigma_i r\sqrt{p_{0,i}}.
\end{equation}
Then $\sum_i(d_0^i)^2=r^2$ and
$(d_0^i)^2/\sum_j(d_0^j)^2=p_{0,i}$.  Conversely, on any fixed open orthant,
the map
\begin{equation}\label{eq:theta}
 \Theta_\sigma:(0,\infty)\times\intt\Delta_{n-1}\to
 \{d\in\R^n:\operatorname{sign}(d_i)=\sigma_i\},
 \qquad
 \Theta_\sigma(r,p)_i=\sigma_ir\sqrt{p_i},
\end{equation}
has the smooth inverse $r=\norm{d}_2$, $p_i=d_i^2/\norm{d}_2^2$, and is therefore a diffeomorphism. In particular, the image of $(1,2)\times P_n$ is a nonempty open set of initial spectral gradients.

Combining Proposition~\ref{prop:n-open} with
Lemma~\ref{lem:projective} gives, coordinate by coordinate,
\[
 \abs{d_0^i}\,10^{-6k}
 \leq
 \abs{d_k^i}
 =\abs{d_0^i}\prod_{t=0}^{k-1}
 \frac{\abs{a_t-\lambda_i}}{a_t}
 \leq
 \abs{d_0^i}\,0.61^k.
\]
This is the spectral version of \eqref{eq:component-main}.
Because $P_n\subset\intt\Delta_{n-1}$, every $d_0^i$ is nonzero.  The lower
bound then makes every $d_k^i$, and hence every $g_k$, nonzero.  Lemma
\ref{lem:bb-step} shows inductively that every BB denominator
$s_{k-1}^Ty_{k-1}=s_{k-1}^TAs_{k-1}$ is strictly positive.  Thus the BB1
iteration is well-defined for all $k$, rather than merely being a formal
solution of the scalar recurrence.  The upper bound tends to zero, so
$g_k\to0$ and, because $A$ is invertible, $x_k-x_*=A^{-1}g_k\to0$.

We next allow arbitrary eigenvectors and prove openness in the full matrix
space.  Define
\begin{equation}\label{eq:An}
 \mathcal A_n:=\{A\in\Spp^n:
  \lambda_1(A)<\cdots<\lambda_n(A),\ 
  (\lambda_1(A),\ldots,\lambda_n(A))\in L_n\}.
\end{equation}
This set is nonempty: it contains the diagonal matrix with any spectrum in
$L_n$.  It is open because the ordered eigenvalues of a real symmetric matrix
are continuous functions of the matrix entries (Weyl's inequality
gives $\abs{\lambda_i(A)-\lambda_i(B)}\leq\norm{A-B}_2$) and $L_n$ is open.

For $A\in\mathcal A_n$ and $g\neq0$, define the orientation-free spectral
weights
\begin{equation}\label{eq:spectral-weights-general}
 p_i(A,g):=\frac{g^T\Pi_i(A)g}{g^Tg}
 =\frac{\norm{\Pi_i(A)g}^2}{\norm{g}^2}.
\end{equation}
They are nonnegative and sum to one because the spectral projectors are
orthogonal and sum to $I$.  

The map $(A,g)\mapsto p(A,g)$ is continuous on simple-spectrum matrices with $g\neq0$: on the simple-spectrum set,
\begin{equation}\label{eq:projector-polynomial}
 \Pi_i(A)=\prod_{j\neq i}
 \frac{A-\lambda_j(A)I}{\lambda_i(A)-\lambda_j(A)}.
\end{equation}
To verify \eqref{eq:projector-polynomial}, apply the polynomial on its right
to an eigenvector of $A$: it equals one on the $i$th eigenspace and zero on
every other eigenspace.  Every denominator is nonzero because the spectrum is simple. Therefore, continuity of the eigenvalues makes the right side locally continuous in $A$, and \eqref{eq:spectral-weights-general} is a quotient of continuous functions with positive denominator $g^Tg$, proving the claim.

Define the open set of matrix--gradient pairs
\begin{equation}\label{eq:Gamma}
 \Gamma_n:=\{(A,g):A\in\mathcal A_n,\ 1<\norm{g}<2,\ p(A,g)\in P_n\}.
\end{equation}
It is open by the continuity just established.  It is nonempty: take a diagonal $A$ with spectrum in $L_n$, any $p\in P_n$, a sign vector $\sigma$, and $g_i=(3/2)\sigma_i\sqrt{p_i}$; then $\norm{g}=3/2$ and $p(A,g)=p$.

For the fixed vector $b$, the map
\begin{equation}\label{eq:joint-diffeo}
 \mathcal J_b:\Spp^n\times\R^n\to\Spp^n\times\R^n,
 \qquad \mathcal J_b(A,x)=(A,Ax-b),
\end{equation}
is a smooth bijection with smooth inverse
$\mathcal J_b^{-1}(A,g)=(A,A^{-1}(g+b))$, matrix inversion being smooth on
$\Spp^n$; it is therefore a diffeomorphism.  Set
\begin{equation}\label{eq:Omega-final}
 \Omega_n:=\mathcal J_b^{-1}(\Gamma_n),
 \qquad
 \mathcal X_n(A):=\{x:(A,x)\in\Omega_n\}.
\end{equation}
The set $\Omega_n$ is nonempty and open, and every fiber
$\mathcal X_n(A)$ is nonempty and open: for a fixed $A\in\mathcal A_n$, one
can choose a spectral-coordinate gradient from \eqref{eq:d-from-p} and then
put $x=A^{-1}(g+b)$.  Every nonempty open subset of a finite-dimensional
Euclidean space has positive Lebesgue measure; applied in dimensions
$n(n+1)/2+n$ and $n$, this proves the positive-measure assertions of
Theorem~\ref{thm:main}.

It remains to prove the norm, objective, and non-superlinear conclusions.
Because $Q$ is orthogonal, \eqref{eq:component-main} gives
\[
 10^{-12k}\norm{g_0}^2
 \leq
 \norm{g_k}^2=\sum_i\abs{d_k^i}^2
 \leq
 0.61^{2k}\norm{g_0}^2,
\]
and taking nonnegative square roots proves \eqref{eq:gradient-main}.
Since $g_k=Ae_k$, one has, in the eigenbasis,
\begin{equation}\label{eq:energy-spectral}
 \norm{e_k}_A^2=g_k^TA^{-1}g_k
 =\sum_{i=1}^n\frac{\abs{d_k^i}^2}{\lambda_i},
\end{equation}
and applying the component bounds term by term proves
\eqref{eq:energy-main}.  Completing the square in \eqref{eq:quadratic} gives
$f(x)-f(x_*)=\tfrac12\norm{x-x_*}_A^2$, so squaring \eqref{eq:energy-main}
and multiplying by $1/2$ proves \eqref{eq:objective-main}.  The lower bounds
in \eqref{eq:root-lower} follow by taking $k$th roots, for example
$\norm{g_k}^{1/k}\geq10^{-6}\norm{g_0}^{1/k}\to10^{-6}$.  Finally,
$Q$-superlinear convergence of a positive null sequence
($u_{k+1}/u_k\to0$) implies $u_k^{1/k}\to0$, so \eqref{eq:root-lower} rules
out both root-superlinear and $Q$-superlinear convergence. All assertions of Theorem~\ref{thm:main} are proved.

\section{Scope and limitations}

The theorem establishes a robust lower-geometric rate on a constructed open
family.  Four qualifications are important.

\begin{enumerate}
\item It is a result for BB1, not BB2.  The initial step is the explicit inverse gradient Rayleigh quotient \eqref{eq:alpha0}.  An arbitrary first step has a related projective description but is not the admissible-initialization map certified here.

\item The conclusion holds on a nonempty open, positive-measure family of spectra and initial points. It is not asserted for every positive spectrum or for almost every spectrum.

\item The computer-assisted step proves existence of one exact orbit inside a tiny rational box. The displayed decimal for $\lambda_3^*$ is a locator for that exact root, not a claim that the eigenvalue is itself a terminating decimal or has a closed form.

\item The constants $10^{-6}$ and $0.61$ are safety margins. The proof uses a minimum transient separation of approximately $1.39\times10^{-4}$ and an upper multiplier bound close to $0.6$ at the central translated spectrum. No optimality is claimed.
\end{enumerate}

\bibliographystyle{plainnat}
\bibliography{bb_references}

\newpage

\appendix

\section{Computer-Assisted Certificate}
\label{app:interval-certificate}

The computer-assisted component of the proof is implemented by the
following Python script. It uses only Python's standard
\texttt{decimal} module.

\begin{Verbatim}
from decimal import Decimal, getcontext, localcontext, ROUND_FLOOR, ROUND_CEILING

getcontext().prec = 120
D = Decimal

def opdown(fn):
    with localcontext() as c:
        c.prec=120; c.rounding=ROUND_FLOOR
        return +fn()
def opup(fn):
    with localcontext() as c:
        c.prec=120; c.rounding=ROUND_CEILING
        return +fn()

class IV:
    __slots__=('lo','hi')
    def __init__(self,lo,hi=None):
        self.lo=D(str(lo)) if not isinstance(lo,D) else lo
        self.hi=self.lo if hi is None else (D(str(hi)) if not isinstance(hi,D) else hi)
        assert self.lo <= self.hi
    def __add__(a,b):
        b=asiv(b); return IV(opdown(lambda:a.lo+b.lo),opup(lambda:a.hi+b.hi))
    __radd__=__add__
    def __neg__(a): return IV(-a.hi,-a.lo)
    def __sub__(a,b): return a+(-asiv(b))
    def __rsub__(a,b): return asiv(b)+(-a)
    def __mul__(a,b):
        b=asiv(b)
        los=[opdown(lambda x=x,y=y:x*y) for x in (a.lo,a.hi) for y in (b.lo,b.hi)]
        his=[opup(lambda x=x,y=y:x*y) for x in (a.lo,a.hi) for y in (b.lo,b.hi)]
        return IV(min(los),max(his))
    __rmul__=__mul__
    def recip(a):
        assert not (a.lo <= 0 <= a.hi)
        valslo=[opdown(lambda x=x:D(1)/x) for x in (a.lo,a.hi)]
        valshi=[opup(lambda x=x:D(1)/x) for x in (a.lo,a.hi)]
        return IV(min(valslo),max(valshi))
    def __truediv__(a,b): return a*asiv(b).recip()
    def __rtruediv__(a,b): return asiv(b)*a.recip()
    def __repr__(a): return f'[{a.lo},{a.hi}]'

def asiv(x): return x if isinstance(x,IV) else IV(x)
def sumup(vals):
    s=D(0)
    for v in vals: s=opup(lambda s=s,v=v:s+v)
    return s

class AD:
    __slots__=('v','d')
    def __init__(self,v,d=None):
        self.v=asiv(v); self.d=[IV(0) for _ in range(8)] if d is None else d
    def __add__(a,b):
        b=asad(b); return AD(a.v+b.v,[x+y for x,y in zip(a.d,b.d)])
    __radd__=__add__
    def __neg__(a): return AD(-a.v,[-x for x in a.d])
    def __sub__(a,b): return a+(-asad(b))
    def __rsub__(a,b): return asad(b)+(-a)
    def __mul__(a,b):
        b=asad(b); return AD(a.v*b.v,[x*b.v+a.v*y for x,y in zip(a.d,b.d)])
    __rmul__=__mul__
    def recip(a): return AD(a.v.recip(),[-x/(a.v*a.v) for x in a.d])
    def __truediv__(a,b): return a*asad(b).recip()
    def __rtruediv__(a,b): return asad(b)*a.recip()

def asad(x): return x if isinstance(x,AD) else AD(x)
def var(box,j):
    der=[IV(0) for _ in range(8)];der[j]=IV(1);return AD(box,der)

L2=AD('1.8786699041860466')
ONE=AD(1); FOUR=AD(4)

def fmap_with_den(z,la):
    m=z[0]; p=[z[1],z[2],z[3],ONE-z[1]-z[2]-z[3]]
    terms=[p[i]*(m-la[i])*(m-la[i]) for i in range(4)]
    den=sum(terms,AD(0))
    mean=sum((la[i]*p[i] for i in range(4)),AD(0))
    return [mean]+[terms[i]/den for i in range(3)],den

def fmap(z,la):
    out,_=fmap_with_den(z,la)
    return out

def fpow(z,la,n):
    for _ in range(n): z=fmap(z,la)
    return z

def f_and_j(boxes):
    x=[var(boxes[i],i) for i in range(8)]
    la=[ONE,L2,x[0],FOUR]
    z=x[1:5]
    a=fpow(z,la,7)
    p=x[5:8]; p4=ONE-sum(p,AD(0))
    m0=sum((la[i]*([p[0],p[1],p[2],p4][i]) for i in range(4)),AD(0))
    b=fpow([m0]+p,la,15)
    out=[a[i]-z[i] for i in range(4)]+[b[i]-z[i] for i in range(4)]
    return [o.v for o in out],[o.d for o in out]

C=[D(s) for s in [
'2.71278149106533722827430219171662596122131200973311943581550524557629575811',
'1.33344967252388129918980034968384933800281196626470897628663867165111253656',
'0.828108598362557830284706279933642675447754907755381063536594022726890227454',
'0.167043051370226631389809462759502319781038286313814177038878187525592790929',
'0.004717957350367412952812533939501994168806798103523267952088050658941184644',
'0.0000134113750442893933116720965327494222816004773939806475307653666721805974245',
'0.989171705769078149047269711515135535847128223427916068474163743526926233731',
'0.00000349260028439145589133467903469364875523517774851271154130219629904643379924']]

RSTR='''
-0.810526348848592888901084652460863262098889862944445915997922,307.242673381334385271539228546921025279578294613703728839309,307.741025629901050999029082768910651032865241495000968766894,312.782742571536891626096323271366901386778871589158283516632,-0.0137368376192333684238175487126343869576306541180311710900804,-421.401532972233614339176178157325006069741606160158648367929,-420.720628332827512110952534958393390951741568565011272813875,-417.510650350147764867647073770427680507322497423363849322798
-2.53381240743594641550513665431248219325091959956406240279772,219.146844709590120107670663557533796892315501194180693740921,221.314903815662600782780829104204280970210726156970789478092,238.433061138882096606510907554404630970281465079411306474734,-0.0066156137536975764499876413028288031850150714975420122664002,-202.94552899548185110702795449214798069667985467445722312406,-202.617607662910774712346795218836979798249918528977551841755,-201.071693306206914696638563015692889394595107198327734687951
1.21923182241183846078543106866408237572050290347089038684799,-283.490269055458601705948958623004565413293682800367378712469,-284.215637538925680074135617356259178167657374408793138161663,-294.814697516715266368336657507713365927445251871977984552338,0.00510430037289669853951692869951609336796693289217752688876014,156.583346292001534198852820229253143092303058113580534397193,156.330337418986871290621900890772993639761529519065013423077,155.137581686689234202670907269428105368197596348629742204233
-1.26108044272173568710731141145266481650790341533600008992776,313.942924362985217723054543998906451668776939851163166266309,314.691544869619096433878614662418685401023724831284472124436,326.12764070745284415061672026001723134524840696137617430052,-0.00557213932631546351238737770698205496104361101774514176877326,-170.935124890502247958801827793221206576836705555576496966117,-170.658926275954081329267222464116649321236505862472492490618,-169.356847511562687722423100405437534964897323010662551564416
0.043078195604279001914097963327737890401646235452894255110503,-31.2979228333921281928882113925871982705271157312435670381131,-31.3219714710519066420310359824923010174212504037635542365452,-32.1662172009335501562092089499137302982952674704895954239567,0.000478099922934842210827835744935335196661930745438524440498953,14.6665517947567543301073202404399084503649597758063756119065,14.642853439673177102278537394045764059087335950989354927027,14.531132658756835739832753728374049724120361151869572985717
0.000226520043874099667313273699369756542509219592106852411297113,-0.0225135215590217304102381601094316680754744631503520593497794,-0.0226850497610954124630365608985981365370252310505830687908151,-0.025225307271391448839622398185040987284482193153534946168519,-0.0000646416290386389765723445987398631521148396996372621461687481,-0.0407002504446717796356496112619123261021708978107613923529244,-0.0408037515085730062610317248805338145956192819108221297272651,-0.0395655532419233021327540805034448469654922061236616605213023
-0.0591009778675161000065292747250136710540662107925102493119043,4.52542863747015400468119868457513686182540510968085333767212,4.5708376271796282978337702766393546845959286011088156970192,5.25277980583025269632142081451745114823486084850385324393674,0.0182267787018476751330387775829333443438138046050174382552172,13.7294182492651964487088572696615033597888320575657007499638,13.7547692717949756360701594372390047794266623618339207810693,13.39537325223256904440264858605473726414587831319750699597
0.0000756712444366593089630182277257919839520913409911366332434105,-0.00819897600424230753264026491865837876618103810542438832290775,-0.00825475135727810105199639024201679021048816336552888167047913,-0.00915735808625744896227380169097910909397341703431881848992986,-0.0000197290639090762745552792690106575683366414360898273244501419,-0.0159791300873645197383620839265853550901431071050090962014869,-0.0160185329798007737064568268873709203041685752673075101630795,-0.0155338627526389510516621091918715132796619616194482426763043'''
R=[[IV(s) for s in line.split(',')] for line in RSTR.strip().splitlines()]

PSTR='''
977656.6634223271235507758960608141,227877753.5872023336353569232362594,226076347.85889685492460652333685333,213421596.35978306381023106533985262
227877753.5872023336353569232362594,53180721626.999672123361693796538579,52760940596.820424133577893141389299,49812202173.83420304678138035243017
226076347.85889685492460652333685333,52760940596.820424133577893141389299,52344478955.779192312153776789975303,49419059318.10867586670666295706405
213421596.35978306381023106533985262,49812202173.83420304678138035243017,49419059318.10867586670666295706405,46657455589.35772965711287496916893'''
P=[[IV(s) for s in line.split(',')] for line in PSTR.strip().splitlines()]

def mm(A,B):
    return [[sum((A[i][k]*B[k][j] for k in range(len(B))),IV(0)) for j in range(len(B[0]))] for i in range(len(A))]
def mv(A,v): return [sum((A[i][j]*v[j] for j in range(len(v))),IV(0)) for i in range(len(A))]
def tr(A): return [list(x) for x in zip(*A)]

def ldl_pivots(A):
    n=len(A); L=[[IV(0) for _ in range(n)] for _ in range(n)]; piv=[]
    for i in range(n):
        L[i][i]=IV(1)
        di=A[i][i]-sum((L[i][k]*L[i][k]*piv[k] for k in range(i)),IV(0))
        piv.append(di)
        for j in range(i+1,n):
            num=A[j][i]-sum((L[j][k]*L[i][k]*piv[k] for k in range(i)),IV(0))
            L[j][i]=num/di
    return piv

if __name__=='__main__':
    rad=D('1e-55')
    boxes=[IV(c-rad,c+rad) for c in C]
    _,J=f_and_j(boxes)
    fx,_=f_and_j([IV(c) for c in C])
    rf=mv(R,fx)
    RJ=mm(R,J)
    B=[[IV(1 if i==j else 0)-RJ[i][j] for j in range(8)] for i in range(8)]
    dx=[IV(-rad,rad) for _ in range(8)]
    bd=mv(B,dx)
    kc=[IV(C[i])-rf[i]+bd[i] for i in range(8)]
    print('Krawczyk ratios |K-c|/radius:')
    for i in range(8):
        assert C[i]-rad < kc[i].lo and kc[i].hi < C[i]+rad
        ratio=max(abs(kc[i].lo-C[i]),abs(kc[i].hi-C[i]))/rad
        print(i,ratio)
    bnorm=max(sumup(max(abs(v.lo),abs(v.hi)) for v in row) for row in B)
    assert bnorm < 1
    print('contraction infinity row-sum bound:',bnorm)
    # Period derivative A = derivative of first four outputs wrt z plus I.
    A=[[J[i][1+j]+(1 if i==j else 0) for j in range(4)] for i in range(4)]
    print('A intervals widths / centers:')
    for row in A: print(' '.join(str((v.lo+v.hi)/2) for v in row))
    print('P LDL lower pivots:')
    ppiv=ldl_pivots(P); assert all(x.lo>0 for x in ppiv)
    print(*(x.lo for x in ppiv))
    Q=[[P[i][j]-mm(mm(tr(A),P),A)[i][j] for j in range(4)] for i in range(4)]
    print('Q diagonal-dominance lower margins:')
    for i in range(4):
        off=sumup(max(abs(Q[i][j].lo),abs(Q[i][j].hi)) for j in range(4) if j!=i)
        margin=opdown(lambda:Q[i][i].lo-off); assert margin>0
        print(margin)
    # Uniform transverse Floquet multiplier for any inserted eigenvalue nu in [3.49,3.51].
    xv=[var(boxes[i],i) for i in range(8)]
    la=[ONE,L2,xv[0],FOUR]; zz=xv[1:5]; nu=AD(IV('3.49','3.51')); tau=AD(1)

    cycle_state=xv[1:5]; cycle_z_min=D('Infinity')
    for _ in range(7):
        cycle_state,den=fmap_with_den(cycle_state,la)
        cycle_z_min=min(cycle_z_min,den.v.lo)
    p=xv[5:8]; p4=ONE-sum(p,AD(0))
    conn_state=[sum((la[i]*[p[0],p[1],p[2],p4][i] for i in range(4)),AD(0))]+p
    conn_z_min=D('Infinity')
    for _ in range(15):
        conn_state,den=fmap_with_den(conn_state,la)
        conn_z_min=min(conn_z_min,den.v.lo)
    print('minimum cycle normalization denominator:',cycle_z_min)
    print('minimum connection normalization denominator:',conn_z_min)
    assert cycle_z_min>D('0.1516357118221540')
    assert conn_z_min>D('0.0009510722601452239')

    cycle_next,_=fmap_with_den(xv[1:5],la)
    phase_gap=(xv[1]-cycle_next[0]).v
    print('first cycle phase difference a0-a1:',phase_gap)
    assert phase_gap.lo>D('0.1782019618055700634')

    for _ in range(7):
        m=zz[0]; pp=[zz[1],zz[2],zz[3],ONE-zz[1]-zz[2]-zz[3]]
        den=sum((pp[i]*(m-la[i])*(m-la[i]) for i in range(4)),AD(0))
        tau=tau*(m-nu)*(m-nu)/den
        zz=fmap(zz,la)
    print('transverse tau [3.49,3.51]:',tau.v)
    assert D(0)<tau.v.lo and tau.v.hi<1
    # No multiplier vanishes during the exceptional initial segment or on the cycle.
    def sep(a,b):
        a=a.v if isinstance(a,AD) else asiv(a); b=b.v if isinstance(b,AD) else asiv(b)
        if a.hi < b.lo: return opdown(lambda:b.lo-a.hi)
        if b.hi < a.lo: return opdown(lambda:a.lo-b.hi)
        return D(0)
    xx=[sum((la[i]*[p[0],p[1],p[2],p4][i] for i in range(4)),AD(0))]+p
    transient_sep=D('Infinity')
    for _ in range(16):
        transient_sep=min(transient_sep,*(sep(xx[0],ell) for ell in la),sep(xx[0],nu))
        xx=fmap(xx,la)
    cycle_sep=D('Infinity'); xx=xv[1:5]
    for _ in range(7):
        cycle_sep=min(cycle_sep,*(sep(xx[0],ell) for ell in la),sep(xx[0],nu))
        xx=fmap(xx,la)
    print('minimum initial-segment m-lambda separation:',transient_sep)
    print('minimum cycle separation (including nu box):',cycle_sep)
    assert transient_sep>0 and cycle_sep>0
\end{Verbatim}

\section{Prompt Used for Proof Verification}
\label{app:audit-prompt}

The following prompt was submitted to GPT 5.6 Sol.

\begin{Verbatim}
Current task statement

For $n\geq4$, given $n$ positive real numbers $\lambda_1\leq ... \leq \lambda_n$. Consider the sequence $\{d^1_k, ..., d^n_k\} (k\geq 0)$ where $d^i_k$ are real numbers for $1\leq i\leq n$, satisfying

$$\begin{equation}
d^i_{k+1}=d^i_k\cdot\left(
\frac{
\sum_{j=1}^n(\lambda_j-\lambda_i)(d^j_{k-1})^2
} {
\sum_{j=1}^n\lambda_j(d^j_{k-1})^2
}
\right), ~~i=1, \cdots, n.
\end{equation}$$
Specifically, in the first iteration, the relation becomes 
$$\begin{equation}
d^i_1=d^i_0\cdot\left(\frac{\sum_{j=1}^n(\lambda_j-\lambda_i)(d^j_0)^2} {\sum_{j=1}^n\lambda_j(d^j_0)^2}\right), ~~i=1, \cdots, n.
\end{equation}$$

Now, given $2n$ real numbers $\lambda_1, ... \lambda_n, d_0^1, ..., d_0^n$, characterize the decreasing rate of the sequence $|d_k^i|$. In particular, prove or disprove the following:

For a positive measure of $(\lambda_1, ..., \lambda_n)\in\mathbb{R}^n$ and a positive measure of $(d_0^1, ..., d_0^n)\in\mathbb{R}^n$, $|d_k^i|$ converges at a linear rate, i.e., there exists positive constants $C_i$ and $0<\rho_i<1$ only related to the $2n$ initial real numbers, such that $|d_k^i|\geq C_i\rho_i^k$.

Assume for purposes of this task that a complete affirmative proof or counterexample exists. Partial progress does not count unless it implies exactly the resolution above. In particular, proofs for zero measure of $(\lambda_1, ..., \lambda_n)\in\mathbb{R}^n$ or zero measure of $(d_0^1, ..., d_0^n)\in\mathbb{R}^n$, $|d_k^i|$, constructions of counter-examples with zero measure of $(\lambda_1, ..., \lambda_n)\in\mathbb{R}^n$ or zero measure of $(d_0^1, ..., d_0^n)\in\mathbb{R}^n$, $|d_k^i|$, reductions to another
unproved conjecture, computational verification through any fixed $n$, and candidate
counterexamples without a complete certificate are insufficient.

Use multiagent v2 aggressively and dynamically. You have up to 64 concurrent agents available. Do not use a fixed assignment such as "N agents for strategy X." Instead, manage the search using the following heuristics:

- Begin with a genuinely diverse portfolio of approaches. Agents should explore substantially different formulations, invariants, reductions, algebraic viewpoints, structural inductions, decompositions, flow formulations, transition systems, embeddings, extremal arguments, and computational sanity checks.

- Do not tell most agents the currently favored approach. Preserve independence during early rounds so that agents do not all converge to the same attractive but incomplete reduction.

- Maintain an explicit registry of approach families. Group agents by the mathematical idea they are using, not by superficial wording. If many agents converge to one family, redirect some of them toward underexplored formulations.

- Do not allow one approach to dominate merely because it gives elegant reductions. A route that ends at a lemma equivalent in strength to the original conjecture is not close to completion unless it supplies a genuinely new proof of that lemma.

- When an approach stalls at a theorem-strength missing lemma, mark that route as blocked. Only continue assigning agents to it if someone proposes a materially new mechanism, invariant, or construction.

- Keep several incompatible proof routes alive through multiple rounds. Cross-pollinate ideas only after independent agents have developed them far enough to expose their real strengths and gaps.

- Require agents to return concrete lemmas, constructions, equations, or counterexamples to proposed sublemmas. Reject status reports, vague optimism, and claims that an unproved global compatibility statement is "routine."

- The root agent should repeatedly synthesize, challenge, redirect, and launch new rounds. Do not stop after the first wave fails. Produce a complete proof if one survives audit; otherwise report only the strongest rigorously proved derivation and its exact remaining gap.

Do not return merely because current approaches fail or agents report theorem-strength gaps.

Continue launching new rounds, reopening blocked approaches only when there is a genuinely new mechanism, and searching for fresh formulations.

Return only when a complete affirmative proof has been found and survives adversarial audit. Do not return a reduction, partial result, isolated missing lemma, "best effort" summary, or
explanation of why the problem is difficult.

Spend at least 8 hours on this before even thinking of returning or giving up.

Public search may be used only for ordinary mathematical background or standard named
theorems, not to search for a solution to this exact conjecture or benchmark. Do not search
the public web merely to determine whether this problem is open, and do not answer that it is open.
\end{Verbatim}
\end{document}